\documentclass[11pt]{amsart}

\usepackage{amsmath, amssymb, xypic}
\newcounter{my_enumerate_counter}
\newcommand{\pushcounter}{\setcounter{my_enumerate_counter}{\value{enumi}}}
\newcommand{\popcounter}{\setcounter{enumi}{\value{my_enumerate_counter}}}

\newcommand{\undefined}{{\tt undefined}}
\newcommand{\liff}{\leftrightarrow}

\DeclareMathOperator{\tp}{tp} 
\DeclareMathOperator{\RegCard}{RegCard}

\newcommand{\bt}{\mathbf t} 
\newcommand{\ba}{\mathbf a} 
\newcommand{\bb}{\mathbf b} 
\newcommand{\bg}{\mathbf g} 
\newcommand{\bd}{\mathbf d} 
\newcommand{\be}{\mathbf e} 
\newcommand{\bh}{\mathbf h}

\newcommand{\bbF}{{\mathbb F}}

\DeclareMathOperator{\Card}{Card}

\newcommand{\bfC}{\mathbf C}

\newcommand{\bbS}{{\mathbb S}}
\newcommand{\bbK}{{\mathbb K}}

\newcommand{\bbN}{{\mathbb N}}
\newcommand{\bbC}{\mathbb C}
\newcommand{\bbQ}{\mathbb Q}
\newcommand{\bbR}{\mathbb R}

\newcommand{\cX}{{\mathcal X}}

\newcommand{\cA}{{\mathcal A}}
\newcommand{\cE}{{\mathcal E}}

\newcommand{\fc}{\mathfrak c} 
\newcommand{\rs}{\restriction}
\newcommand{\ff}{f} 
\newcommand{\cT}{\mathcal T}

\DeclareMathOperator{\FI}{FI}

\DeclareMathOperator{\pcU}{\prod_{\cU}}
\DeclareMathOperator{\cU}{\mathcal U} 

\DeclareMathOperator{\cV}{\mathcal V} 
\DeclareMathOperator{\dom}{dom}

\newcommand{\cC}{\mathcal C}

\newcommand{\cG}{\mathcal G}
\newcommand{\ccG}{\mathcal G}
\newcommand{\cB}{\mathcal B}

\newcommand{\calD}{\mathcal D}

\newcommand{\e}{\varepsilon}
\newtheorem{thm}{Theorem}[section]
\newtheorem{theorem}{Theorem}

\newtheorem{corollary}[theorem]{Corollary}

\newtheorem{claim}[thm]{Claim}

\newtheorem{lemma}[thm]{Lemma}

\newtheorem{prop}[thm]{Proposition}

\theoremstyle{definition}
\newtheorem{remark}[thm]{Remark}

\newtheorem{definition}[thm]{Definition}

\newtheorem{example}[thm]{Example}
\DeclareMathOperator{\inv}{inv} 
\DeclareMathOperator{\INV}{INV} 
\DeclareMathOperator{\cf}{cf} 

\newcommand{\cP}{\mathcal P} 
\renewcommand{\vec}{\bar}
\DeclareMathOperator{\Alt}{Alt}

\title{A dichotomy for the number of ultrapowers}

\author{Ilijas Farah}

\address{Department of Mathematics and Statistics\\
York University\\
4700 Keele Street\\
North York, Ontario\\ Canada, M3J
1P3\\
and Matematicki Institut, Kneza Mihaila 35, Belgrade, Serbia}

\urladdr{http://www.math.yorku.ca/$\sim$ifarah}

\email{ifarah@mathstat.yorku.ca}

\author{Saharon Shelah}

\address{The Hebrew University of Jerusalem\\
Einstein Institute of Mathematics\\
Edmond J. Safra Campus, Givat Ram\\
Jerusalem 91904, Israel\\ and \\
Department of Mathematics\\
Hill Center-Busch Campus\\
Rutgers, The State University of New Jersey\\
110 Frelinghuysen Road\\
Piscataway, NJ 08854-8019 USA}

\email{shelah@math.huji.ac.il}
\urladdr{http://shelah.logic.at/}
\thanks{The first author was  partially supported by NSERC and he would like to thank 
Takeshi Katsura for several useful remarks}
\thanks{The second author would like to thank the Israel Science Foundation for
  partial support of this research (Grant no. 710/07).
Part of this work was done when the authors visited the Mittag-Leffler Institute. No. 954 on 
Shelah's list of publications. }

\date{\today}

\begin{document}

\begin{abstract} 
We prove a strong dichotomy for the number of ultrapowers 
of a given  model of cardinality $\leq 2^{\aleph_0}$  associated with nonprincipal 
ultrafilters on~$\bbN$. They are either all isomorphic, or
else there are $2^{2^{\aleph_0}}$ many nonisomorphic ultrapowers. 
We prove the analogous result for metric structures, including C*-algebras 
and II$_1$ factors, as well as their relative commutants and include 
several applications. 
We also show that the C*-algebra~$\cB(H)$ always has nonisomorphic relative commutants 
in its ultrapowers associated with nonprincipal ultrafilters on~$\bbN$.
\end{abstract} 

\maketitle


\section{Introduction} 
In the following all ultrafilters are nonprincipal ultrafilters on $\bbN$. 
In particular, `all ultrapowers of $A$' always stands for `all ultrapowers associated with 
 nonprincipal ultrafilters on $\bbN$.'

The question of counting the number of nonisomorphic models of a given theory in a given 
cardinality 
was one of the main driving forces behind the development of  
Model Theory (see Morley's Theorem and \cite{Sh:c}). On the other hand, the question of counting 
the 
number of nonisomorphic ultrapowers of a given model has received more attention from
 functional analysts than from  
logicians. 

Consider a countable structure $A$ in a countable signature. By a classical result of Keisler, 
every ultrapower $\prod_{\cU} A$ is countably saturated (recall that~$\cU$ is assumed 
to be a nonprincipal ultrafilter on $\bbN$). This implies that the ultrapowers of $A$ 
are not easy to distinguish. Moreover, 
if the Continuum Hypothesis holds then they are all saturated
and therefore isomorphic (this fact will not be used in the present paper; see \cite{ChaKe}). 

Therefore the question of counting nonisomorphic ultrapowers of a given countable 
structure is nontrivial only when the Continuum Hypothesis fails, and  
in the remaining 
part of this introduction we assume that it does fail. 
If we moreover assume that the theory of $A$ is unstable (or equivalently, 
that it 
has the \emph{order property}---see
the beginning of \S\ref{S.Representing}) then $A$ has nonisomorphic 
ultrapowers (\cite[Theorem~VI.3]{Sh:c} and independently \cite{Do:Ultrapowers}). 
The converse, that if the theory of $A$ is stable then all of its ultrapowers are isomorphic, 
was proved only recently (\cite{FaHaSh:Model2}) although  main components of the 
proof were present in \cite{Sh:c} and the result was essentially known to the second author. 

The question of the isomorphism of ultrapowers was first asked
by operator algebraists. This is not so surprising in the light of the fact that the ultrapower
construction is an indispensable tool in Functional Analysis and in particular in Operator
Algebras. The  ultrapower construction for Banach spaces, 
C*-algebras, or II$_1$ factors is again 
an honest metric structure of the same type. These constructions 
coincide with the ultrapower construction for metric structures as defined
in \cite{BYBHU} (see also \cite{FaHaSh:Model2}). 
The Dow--Shelah result can be used to prove  that C*-algebras and II$_1$ factors 
have nonisomorphic ultrapowers (\cite{GeHa} and \cite{FaHaSh:Model1}, respectively), 
and with some extra effort this conclusion can be extended to the relative commutants
of separable
C*-algebras and II$_1$ factors in their utrapowers (\cite{Fa:Relative} and \cite{FaHaSh:Model1}, 
respectively). 

However, the methods used in \cite{GeHa}, \cite{Fa:Relative} and \cite{FaHaSh:Model1} 
  provide only as many nonisomorphic
ultrapowers as there are uncountable cardinals $\leq\fc=2^{\aleph_0}$ (with our assumption, 
 two). 
In  \cite[\S3]{KShTS:818} it was proved (still assuming only that CH fails) 
that $(\bbN,<)$ has $2^{\fc}$ nonisomorphic ultrapowers. As pointed out
in \cite{ElHaScTh}, this proof could easily be modified to obtain the same conclusion 
for any infinite linear (sometimes called \emph{total}) order in place of $(\bbN,<)$
but the proof  does not cover  even  the case of an arbitrary 
 partially ordered set with an infinite chain.

\begin{theorem}\label{T1}  Assume the Continuum Hypothesis, CH,  fails. 
If $A$ is a model of cardinality $\leq\fc$ 
such that the theory of $A$ is unstable, then 
there are $2^{\mathfrak c}$ isomorphism types of models of the form $\pcU A$, 
where $\cU$ ranges over nonprincipal ultrafilters on~$\bbN$. 
\end{theorem}

In Theorem~\ref{T1+} we prove a generalization of Theorem~\ref{T1} for ultraproducts. 

\begin{corollary} \label{C1} 
For a model $A$ of cardinality $\leq\fc$ with a countable signature 
either all of its ultrapowers are isomorphic or there are $2^{\mathfrak c}$ isomorphism
types of its ultrapowers. 
\end{corollary}

\begin{proof} We may assume $A$ is infinite. 
If the theory of $A$ is stable, then $\pcU A$ is  saturated and of cardinality $\fc$ and therefore 
all such ultrapowers are isomorphic (\cite{FaHaSh:Model2}). 
If the Continuum Hypothesis holds, then all the ultrapowers are isomorphic
by Keisler's result. In the remaining case when  the Continuum Hypothesis fails and  
the theory of $A$ is unstable use Theorem~\ref{T1}. 
\end{proof} 

We also prove the analogue  of Theorem~\ref{T1} for metric structures
(see \cite{BYBHU} or~\cite{FaHaSh:Model2}). 
The ultrapowers of metric structures are defined in \S\ref{S.metric}. 
Recall that the \emph{character density} of a metric space is the 
minimal cardinality of its dense subspace.

\begin{theorem}\label{T1.m}  Assume CH fails. 
If $A$ is a metric structure of character density $\leq\fc$ 
such that the theory of $A$ is unstable, then 
there are $2^{\mathfrak c}$ isometry  types of models of the form $\pcU A$, 
where $\cU$ ranges over nonprincipal ultrafilters on~$\bbN$. 
\end{theorem} 

The proof is a modification of the proof of Theorem~\ref{T1} 
and it will be outlined in \S\ref{S.metric}. 
Although Theorem~\ref{T1.m} implies Theorem~\ref{T1}, we chose to present
the proof of Theorem~\ref{T1} separately because it is the main case and 
because some of the main ideas are more transparent in the discrete case.

\begin{corollary} \label{C1.m} 
For a metric structure  $A$ of character density  $\leq\fc$ with a countable signature 
either all of its ultrapowers are isomorphic or there are~$2^{\mathfrak c}$ isomorphism
types of its ultrapowers. 
\end{corollary}

\begin{proof} We may assume $A$ is infinite. 
If the theory of $A$ is stable, then $\pcU A$ is  saturated and of character density 
 $\fc$ and therefore 
all such ultrapowers are isomorphic (\cite{FaHaSh:Model2}). 
If the Continuum Hypothesis holds then all ultrapowers are 
isomorphic by the analogue of Keisler's theorem for metric 
structures (\cite{BYBHU}). In the remaining case, when the 
Continuum Hypothesis fails and 
 the theory of $A$ is unstable use Theorem~\ref{T1.m}. 
\end{proof} 

Important instances  of the ultraproduct construction for metric spaces include C*-algebras, 
II$_1$ factors (see e.g., \cite{FaHaSh:Model2}) and metric groups (see \cite{Pe:Hyperlinear}).

\subsection*{Organization of the paper} 
The proof of Theorem~\ref{T1}  uses ideas from  \cite[\S VI.3]{Sh:c},  \cite[\S3]{KShTS:818}
and \cite[III.3]{Sh:e} and it  
will be presented in \S\ref{S.Invariants}, \S\ref{S.Representing}, 
\S\ref{S.Construction} and \S\ref{S.proof.T1}. 
Theorem~\ref{T1.m} is 
proved in~\S\ref{S.metric}, 
and some applications will be given in \S\ref{S.Applications}.
 In \S\ref{S.Local} we prove  local versions of Theorem~\ref{T1} and Theorem~\ref{T1.m}, 
 and in Proposition~\ref{P.B(H)} we use the latter to prove that $\cB(H)$ always 
 has nonisomorphic relative commutants in  its ultrapowers associated with nonprincipal 
 ultrafilters on~$\bbN$.   
Sections \S\ref{S.Invariants} and \S\ref{S.Representing} are essentially a revision
of \cite[\S 3]{Sh:c}, and \S\ref{S.Construction} has a small, albeit nonempty intersection 
with~\cite[\S3]{KShTS:818} (and therefore with the latter half of \cite[\S VI.3]{Sh:c}). 

\subsection*{Notation and terminology} If $A$ denotes a model, then its universe is 
also denoted by $A$ and the cardinality of its universe (or any other set $A$) is denoted
by $|A|$. Hence what we denote by $A$ is denoted by  $A$ or by  $|A|$  in \cite{Sh:c} and 
\cite{Sh:e}, and what we denote by $|A|$ is denoted by $||A||$ in \cite{Sh:c} and \cite{Sh:e}
if $A$ is a model. 
We also don't distinguish the notation for a formula $\phi(x)$ and its evaluation $\phi[a]$ in a
model. It will always be clear from the context.

Letters $I$ and $J$, possibly with subscripts or superscripts, will always denote \emph{linear}  
(i.e., \emph{total}) orders. 
The reverse of a linear 
order $I$ will be denoted by $I^*$. The \emph{cofinality} of a 
linear order $I$, $\cf(I)$, is the mininal cardinality of a cofinal subset of $I$. 
By $I+J$ we denote the order with domain $I\sqcup J$ in which
copies of $I$ and $J$ are taken with the original ordering and 
 $i<j$ for all $i\in I$ and all $j\in J$. 
If $J$ and $I_j$, for $j\in J$, are linear orders then $\sum_{j\in J} I_j$ denotes
the order with the underlying set $\bigcup_{j\in J} \{j\}\times I_j$ ordered lexicographically.

Following the notation common in Model Theory, 
an ultrapower of $A$ associated with an ultrafilter $\cU$ will be denoted by $\prod_{\cU} A$, 
even in the case when $A$ is an operator algebra, where the notation $A^{\cU}$ for the
ultrapower is standard. We refrain from using the symbol $\omega$ in order to avoid confusion.

By $\forall^\infty m$ we denote the quantifier `for all large enough $m\in \bbN$.' 
More generally, if  $D$ is a filter on $\bbN$ then by  $(\forall^D n)$ we denote the quantifier
as a shortcut for `the set of all $n$ such that\dots{} belongs to $D$.' 

An $n$-tuple of elements of $A$ is always denoted by $\vec a$. 
 
For $k\geq 1$ by $[X]^k$ we denote the set of all $k$-element subsets of $X$. 

A cardinal $\kappa$ will be identified with the least ordinal of cardinality $\kappa$, 
as well as the linear order~$(\kappa,<)$. A cardinal $\kappa$ is \emph{regular} 
if $\kappa=\cf(\kappa)$ and \emph{singular} otherwise. 
An increasing family of ordinals or cardinals $\lambda_\xi$, for $\xi<\gamma$, is
\emph{continuous} if $\lambda_\eta=\sup_{\xi<\eta} \lambda_\xi$ whenever 
$\eta$ is a limit ordinal. 
Analogously, an increasing family $A_\xi$, for $\xi<\gamma$, of sets is \emph{continuous} 
if $A_\eta=\bigcup_{\xi<\eta} A_\xi$ for every limit ordinal $\eta$.

\section{Invariants of linear orders} 
\label{S.Invariants} 
The material of the present and the following 
sections is loosely based on \cite[III.3]{Sh:e}.

\subsection{The invariant $\inv^m(J)$} \label{S.m}
In the following we consider the invariant $\inv^\alpha_\kappa(I)$ as defined in 
 \cite[Definition~III.3.4]{Sh:e}, or rather its special case when $\alpha=m\in\bbN$ and 
 $\kappa=\aleph_1$. All the arguments presented here can straightforwardly be extended
 to the more general context of an arbitrary ordinal $\alpha$ and regular cardinal $\kappa$. 

In certain cases we define the invariant to be \undefined{}. The phrase  
 `an invariant is defined' will be used as an abbreviation for `an invariant is not equal to 
 \undefined{}.'
  
For a linear order $(I,\leq)$ define  $\inv^m(I)$, for $m\in \bbN$,  by recursion as follows. 
If $\inv^m(I)$ is \undefined{} for some $m$, then $\inv^{m+1}(I)$ is also \undefined. 
If $\cf(I)\leq \aleph_0$ then let $\inv^0(I)$ be \undefined. 
Otherwise let 
\[
\inv^0(I)=\cf(I). 
\]
In order to define $\inv^m(I)$ for $m\geq 1$ write $\kappa=\inv^0(I)$. Although the definition 
when $m=1$ is a special case of  the general case, we single it out as a warmup.    
Fix a continuous sequence $I_\xi$, for $\xi<\kappa$, of proper initial segments of $I$ 
such that $I=\bigcup_{\xi<\kappa} I_\xi$. Then let $\lambda_\xi=\cf((I\setminus I_\xi)^*)$, 
where $J^*$ denotes the reverse order on $J$. Thus $\lambda_\xi$, for $\xi<\kappa$, 
is the sequence of coinitialities of end-segments of $I$ corresponding to the sequence 
$I_\xi$, for $\xi<\kappa$. 

Let $\calD(\kappa,\aleph_1)$  be the filter  on $\kappa$ dual to the ideal 
generated by the nonstationary ideal and 
the set $\{\xi<\kappa\colon \cf(\xi)\leq \aleph_0\}$. 
Define $f\colon \kappa\to \Card$~by 
\[
f(\xi)=
\begin{cases}
 \lambda_\xi, & \text{ if } \lambda_\xi\geq \aleph_1\\
0 & \text{ if } \lambda_\xi\leq \aleph_0.
\end{cases} 
\]
If the set $\{\xi\colon f(\xi)=0\}$ belongs to $\calD(\kappa,\aleph_1)$ then 
let $\inv^1(I)$ be the equivalence class of $f$ modulo $\calD(\kappa,\aleph_1)$, 
or in symbols 
\[
\inv^1(I)=f/\calD(\kappa,\aleph_1). 
\] 
Otherwise, $\inv^1(I)$ is \undefined{}. 

Assume $m\geq 1$ and $\inv^m(J)$ is defined for all linear orders $J$ (allowing
the very definition of $\inv^m(J)$ to be `\undefined{}'). 
Assume $I$ and $I_\xi$,  
for $\xi<\kappa=\cf(I)$, are as in the case $m=1$. 
Define a
function $g_{m}$ with domain~$\kappa$ via
\[
g_{m}(\eta)=\inv^m((I\setminus I_\eta)^*). 
\]
If $\{\eta\colon g_{m}(\eta)$ is defined$\}$ belongs 
to $\calD(\kappa,\aleph_1)$ then  let 
$\inv^{m+1}(I)$ be the equivalence class of $g_{m}$ modulo  $\calD(\kappa,\aleph_1)$. 
Otherwise $\inv^{m+1}(I)$ is \undefined{}.

This defines $\inv^m(I)$ for all $I$. 
For a (defined)  invariant $\bd$ we shall write $\cf(\bd)$ for $\cf(I)$, where $I$ is any linear
order with $\inv^m(I)=\bd$. We also write 
\[
|\bd|=\min\{|I| \colon \bd=\inv^m(I)\text{ for some $m$}\}.
\]
  Our invariant $\inv^m(I)$ essentially corresponds to $\inv^m_{\aleph_1}(I)$ as defined
in \cite[Definition~III.3.4]{Sh:e}.
 Although $\inv^\eta$ can be recursively 
defined for every ordinal~$\eta$, we do not have applications for this general notion. 
As a matter of fact, only $\inv^m$ for  $m\leq 3$
will be used in the present paper.

\begin{example} \label{Example} 
Assume throughout this example that  $\kappa$ is 
a cardinal with  $\cf(\kappa)\geq \aleph_1$. 

(1) Then $\inv^0(\kappa)=\cf(\kappa)$ and $\inv^1(\kappa)$ is \undefined{}. 

(2) If $\lambda$ is a cardinal with $\cf(\lambda)\geq \aleph_1$ 
then $\inv^0(\kappa\times\lambda^*)=\cf(\kappa)$ and $\inv^1(\kappa\times\lambda^*)$
is the equivalence class of the function on $ \cf(\kappa)$ everywhere equal to $\cf(\lambda)$, 
modulo the ideal $\calD(\cf(\kappa),\aleph_1)$. 

(3) If $\inv^m(I_\xi)$ is defined for all $\xi<\kappa$ and $\kappa$ is regular
then with $I=\sum_{\xi<\kappa} I_\xi^*$
we have that $\inv^{m+1}(I)$ is the equivalence class of the function $g(\xi)=\inv^m(I_\xi)$
modulo $\calD(\kappa,\aleph_1)$. 
\end{example}

Example (3) above will be used to define linear orders with prescribed invariants.


\begin{lemma} \label{L.lo} 
\begin{enumerate}
\item For every regular $\lambda\geq\aleph_2$ 
 there are $2^{\lambda}$ linear orders of cardinality $\lambda$
with pairwise distinct, defined, invariants $\inv^1(I)$. 
\item 
If $\lambda$ is singular  then for every regular uncountable $\theta$ such that 
\[
\max(\aleph_2,\cf(\lambda))\leq\theta<\lambda
\]  
 there are $2^\lambda$ linear orders of cardinality $\lambda$ and cofinality $\theta$
with pairwise distinct, defined, invariants $\inv^2(I)$. 
\end{enumerate}
\end{lemma} 

\begin{proof} This is cases (1--3) of 
 \cite[Lemma~III.3.8]{Sh:e}, with  $\kappa=\aleph_1$ 
but we reproduce the proof for the convenience of the reader. 

(1) If $\lambda\geq\aleph_2$ is regular, then
the set $\{\xi<\lambda\colon \cf(\xi)\geq \aleph_1\}$ can be partitioned into $\lambda$ disjoint
stationary sets (see \cite[Appendix, Theorem~1.3(2)]{Sh:c} or 
\cite[Corollary~6.12]{Ku:Book}). 
Denote these sets by  $S_\eta$, for $\eta<\lambda$. 
For $Z\subseteq \lambda$ define a linear order 
$L_Z$ as follows. For $\alpha<\lambda$ let 
\[
\kappa(\alpha)=\begin{cases} 
\aleph_1&\text{ if }\alpha\in \bigcup_{\eta\in Z} S_\eta\\
\aleph_2&\text{ if }\alpha\in \bigcup_{\eta\notin Z} S_\eta\\
1&\text{ if } \cf(\alpha)\leq\aleph_0.
\end{cases}
\]
Let $L_Z=\sum_{\alpha<\lambda}\kappa(\alpha)^*$. More formally, 
let the domain of $L_Z$ be the set 
$\{(\alpha,\beta)\colon \alpha<\lambda, \beta<\kappa(\alpha)\}$
ordered by $(\alpha_1,\beta_1)\prec_L (\alpha_2,\beta_2)$ if $\alpha_1<\alpha_2$ or
$\alpha_1=\alpha_2$ and $\beta_1>\beta_2$. 
Then $\inv^1(L_Z)$ is clearly defined. 
A standard argument using the stationarity of $S_\xi$ for any $\xi\in Z\Delta Y$ 
shows that $\inv^1(L_Z)\neq \inv^1(L_Y)$ if $Z\neq Y$. 

(2) Now assume $\lambda$ is singular. 
Pick regular cardinals
 $\lambda_i$, for $i<\cf(\lambda)$, such that $\sum_{i<\lambda} \lambda_i=\lambda$. 
Using (1) for each $i$ fix linear orders $I_{ij}$, for $j<2^{\lambda_i}$, of cardinality $\lambda_i$
such that $\inv^1(I_{ij})$ are all defined and distinct. 
Since $|\prod_{i<\cf(\lambda)} 2^{\lambda_i}|=2^\lambda$ it will suffice to 
associate a linear order $J_g$ to every $g\in \prod_{i<\cf(\lambda)} 2^{\lambda_i}$ such that 
$\inv^2(J_g)$ is defined for every $g$ and $\inv^2(J_g)\neq \inv^2(J_h)$ whenever $g\neq h$. 

Since $\theta\geq\max(\aleph_2,\cf(\lambda))$, by 
 \cite[Appendix, Theorem~1.3(2)]{Sh:c} or 
\cite[Corollary~6.12]{Ku:Book} we may partition  the set 
 $\{\xi<\theta\colon \cf(\xi)\geq \aleph_1\}$ into $\cf(\lambda)$ stationary sets $S_\xi$, for $\xi<
\cf(\lambda)$. 
Then 
\[
\textstyle J_g=\sum_{\xi<\theta} I_{\xi,g(\xi)}^*
\]
has $\inv^0(J_g)=\theta$ and 
$\inv^2(J_g)=\langle \inv^1(I_{\xi,g(\xi)})\colon \xi<\theta\rangle/\calD(\theta,\aleph_1)$. 
If $\xi$ is such that 
 $h(\xi)\neq g(\xi)$ then the representing sequences of $\inv^2(J_g)$ and $\inv^2(J_h)$
disagree on the stationary set $S_\xi$. Therefore $g\mapsto \inv^2(J_g)$ is an injection, as 
required.  
\end{proof}

\subsection{A modified  invariant $\inv^{m,\lambda}(J)$}
\label{S.m,lambda} 
Fix a cardinal $\lambda$.  For  a linear order~$J$ of cardinality $\lambda$ and $m\in\bbN$
we define an invariant that is a modification of $\inv^m(J)$, considering three cases. 
Recall that for a regular cardinal $\lambda$ we let $\calD(\lambda,\aleph_1)$ denote   filter
on $\lambda$ generated by the club filter and $\{\xi<\lambda\colon \cf(\xi)\geq\aleph_1\}$. 

\subsubsection{Assume $\lambda$ is regular.} Then let 
$\inv^{m,\lambda}(J)=\inv^m(J)$
if $\cf(J)=\lambda$ and \undefined{} otherwise. 

\subsubsection{Assume $\lambda$ is singular and $\cf(\lambda)>\aleph_1$} 
\label{S.2.2.2} 
Fix an increasing  continuous sequence of 
cardinals $\lambda_\xi$, for $\xi<\cf(\lambda)$, such that
$\lambda=\sup_{\xi<\cf(\lambda)}\lambda_\xi$.  

Then let 
$\inv^{0,\lambda}(J)=\inv^0(J)$ if $\cf(J)=\cf(\lambda)$ and \undefined{} otherwise. 
If $m\geq 1$ and $\inv^{0,\lambda}(J)$ is defined, then let $\inv^{m,\lambda}(J)=\inv^m(J)$ 
if $\inv^m(J)=\langle \bd_\xi\colon \xi<\cf(\lambda)\rangle$ is such that  
\[
\{\xi<\cf(\lambda)\colon \cf(\bd_\xi)>\lambda_\xi\}\in \calD(\cf(\lambda),\aleph_1). 
\]

\subsubsection{Assume $\lambda$ is singular and $\aleph_1\geq\cf(\lambda)$}
This case will require extra work. 
Like above, fix an increasing  continuous sequence of 
cardinals $\lambda_\xi$, for $\xi<\cf(\lambda)$, such that
$\lambda=\sup_{\xi<\cf(\lambda)}\lambda_\xi$.  
By $\RegCard$ we denote the class of all regular cardinals. 

\begin{lemma} \label{L.h} If $\cf(\lambda)\leq \aleph_1$ 
then there is  $h=h_\lambda\colon \aleph_2\to \lambda\cap \RegCard$  
  such that 
$h^{-1}([\mu,\lambda))$ is $\calD(\aleph_2,\aleph_1)$-positive for every $\mu<\lambda$. 
\end{lemma} 

 \begin{proof} Partition $\aleph_2$ into $\cf(\lambda)$ sets  $S_\xi$, $\xi<\cf(\lambda)$
 that are $\calD(\aleph_2,\aleph_1)$-positive and pick $h(\xi)>\eta$ if $\xi\in S_\eta$. 
 \end{proof} 
 
 With $h=h_\lambda$ as in Lemma~\ref{L.h} let  $\calD_h(\aleph_2)$
be the filter generated by $\calD(\aleph_2,\aleph_1)$ and the sets $h^{-1}([\mu,\lambda))$
for $\mu<\lambda$.  In the following the function $h_\lambda$ will be fixed 
for each $\lambda$ such that $\cf(\lambda)\leq\aleph_1$. We shall therefore 
suppress writing $h$ everywhere except in $\calD_{h_\lambda}(\aleph_2)$, usually 
dropping the subscript $\lambda$ which will be clear from the context.

Define $\inv^{m,\lambda}(J)$ (really $\inv^{m,\lambda,h}(J)$) as follows. 

Let $\inv^{0,\lambda}(J)=\inv^0(J)$ if $\cf(J)=\aleph_2$ and \undefined{} otherwise.  

Assume $m\geq 1$ and 
\[
\inv^m(J)=\langle \bd_\xi\colon \xi<\aleph_2\rangle/\calD(\aleph_2,\aleph_1). 
\]
If $\{\xi\colon \cf(\bd_\xi)>\lambda_{h(\xi)}\}\in \calD_h(\aleph_2)$ 
then  let 
 \[
 \inv^{m,\lambda}(J)=
 \langle \bd_\xi\colon \xi<\aleph_2\rangle/\calD_h(\aleph_2)
\]
and \undefined{} otherwise. 

Since $\calD_h(\aleph_2)$ extends $\calD(\aleph_2,\aleph_1)$, this invariant  is well-defined. 

\begin{definition} Given a cardinal $\lambda\geq\aleph_2$ and $m\in\bbN$, 
an \emph{$m,\lambda$-invariant} is any  invariant  $\inv^{m,\lambda}(J)$
for a linear order $J$ of cardinality $\lambda$ that is not equal to \undefined. 
\end{definition} 

Two representing sequences  $\langle \bd_\xi\colon \xi<\kappa\rangle$ and 
$\langle\be_\xi\colon \xi<\kappa\rangle$ of invariants of the same cofinality
 $\kappa$ are \emph{disjoint} if $\bd_\xi\neq \be_\xi$ for all $\xi$. 
 Note that this is not a property of the invariants since it depends on the choice of the 
 representing sequences. 

\begin{lemma}\label{L.many-invariants} For every cardinal $\lambda\geq\aleph_2$ there 
exist $m\in\bbN$ and 
 $2^\lambda$ disjoint representing sequences of $m,\lambda$-invariants of linear
orders of cardinality $\lambda$. 
\end{lemma} 

\begin{proof} Assume first $\lambda$ is regular. 
By Lemma~\ref{L.lo} there are $2^\lambda$ linear 
orders of  cardinality $\lambda$ andÊwith 
 cofinality equal to   $\lambda$, listed as $I_\xi$ for $\xi<2^\lambda$, 
with distinct (and defined) invariants $\inv^1(I_\xi)$. 
Let $I_\xi=\lambda\times J_\xi^*$.  Then $|I_\xi|=\lambda$, 
$\inv^{2,\lambda}(I_\xi)$ 
is defined since $\cf(I_\xi)=\lambda$ for all $\xi$ and it has constant representing sequence. 
Therefore all these representing sequences are disjoint. 

Now assume $\lambda$ is singular. 
By Lemma~\ref{L.lo} for every regular $\theta<\lambda$ there are
$2^\lambda$ linear orders, $J_{\theta,\xi}$, for $\xi<2^\lambda$, 
of cardinality $\lambda$, cofinality $\theta$, and with distinct and defined
invariants $\inv^2(J_{\theta,\xi})$. 

(a)  Assume furthermore that   $\cf(\lambda)\geq\aleph_2$.  Fix an increasing
continuous sequence  $\lambda_\eta$, 
for $\eta<\cf(\lambda)$ with the supremum equal to $\lambda$, 
as in \S\ref{S.2.2.2}.   Now fix an increasing  
sequence $\theta_\eta$, for $\eta<\cf(\lambda)$, of regular cardinals 
with the supremum equal to~$\lambda$ and such that $\theta_\eta>\lambda_\eta$ for all $\eta$. 
For $\xi<2^\lambda$ let 
\[
\textstyle I_\xi=\sum_{\eta<\cf(\lambda)} I_{\theta,\xi}{}^*
\]
(see Example~\ref{Example} (3)). 
Then each linear order $I_\xi$, for $\xi<2^\lambda$, has cardinality~$\lambda$, 
$\inv^{3,\lambda}(I_\xi)$ is defined for all $\xi$,
and the obvious  representing sequences for $\inv^{3,\lambda}(I_\xi)$  are disjoint.

(b) Now assume  $\cf(\lambda)\leq\aleph_1$ and 
consider $h=h_\lambda\colon \aleph_2\to \lambda\cap\RegCard$ as in Lemma~\ref{L.h}.
For $\xi<2^\lambda$ let 
$
\textstyle I_\xi=\sum_{\eta<\aleph_2} I_{h(\eta),\xi}{}^*$.
Then each linear order $I_\xi$, for $\xi<2^\lambda$, has cardinality $\lambda$,
$\inv^{3,\lambda}(I_\xi)$ is defined,
and the obvious  representing sequences for $\inv^{3,\lambda}(I_\xi)$  are disjoint. 
\end{proof}

\section{Representing invariants in models of theories with the order property} 
\label{S.Representing} 

\subsection{The order property} \label{S.OP} 
In the present section $A$ is a model of  countable signature whose theory 
has the \emph{order property},  
as witnessed by  formula $\phi(\vec x,\vec y)$. Thus there is  $n\geq 1$ such that
$\phi$ is a $2n$-ary formula and
in $A^n$ there exist arbitrarily long finite $\prec_\phi$ chains, where $\prec_\phi$ is a binary 
relation on 
$A^n$ defined by letting $\vec a\prec_\phi \vec b$  if 
\[
A\models \phi(\vec a,\vec b) \land \lnot \phi(\vec b, \vec a). 
\]
It should be emphasized that $\prec_\phi$ is not required to be transitiive. 

The existence of such formula $\phi$ is equivalent to the theory of $A$ being unstable
(\cite[Theorem~2.13]{Sh:c}). This fact 
is the only bit of stability theory needed in the present paper. 

We shall  write $A\models\vec a\preceq_\phi \vec b$ to signify that 
 $A\models \vec a\prec_\phi \vec b$ or $A\models \vec a=\vec b$. 
We shall frequently write $\vec a\prec_\phi \vec b$ and $\vec a\preceq_\phi \vec b$ 
instead of $A\models \vec a\prec_\phi \vec b$ and $A\models \vec a\preceq_\phi\vec b$
since at any given instance we will 
deal with a fixed $A$ and its elementary substructures.

A \emph{$\phi$-chain} is a subset of  $A^n$ linearly ordered by $\preceq_\phi$. 
For $\vec b$ and $\vec c$ in $A^n$ we write 
\[
[\vec b,\vec c]_\phi=\{\vec d\colon 
 \vec b\preceq_\phi \vec d\land  \vec d\preceq_\phi \vec c\}
\]
and similarly 
\begin{align*}
(-\infty,\vec c]&=\{\vec d\colon  \vec d\preceq_\phi \vec c\},
\text{ and} \\
[\vec c,\infty)&=\{\vec d\colon  \vec c\preceq_\phi \vec d\}.
\end{align*}
If $\cC$ is a $\phi$-chain in $A$ then we shall freely use phrases 
such as `large enough $\vec c\in\cC$' with their obvious meaning. 
By $\cf(\cC)$ we denote the cofinality of $(\cC,\preceq_\phi)$. 
We shall sometimes consider $\phi$-chains with the reverse ordering, $\preceq_{\lnot\phi}$. 
Whenever deemed necessary this will be made explicit by writing $(\cC,\preceq_{\lnot\phi})$
as in e.g., $\cf(\cC,\preceq_{\lnot\phi})$. 
Since $\preceq_\phi$ need not be transitive, one has to use this notation with some care.

\subsection{Combinatorics of the invariants} 

The following is a special case of the definition of `weakly $(\kappa,\Delta)$-skeleton like'
where $\kappa$ is an arbitrary cardinal and $\Delta$ is set of 
formulas as given in \cite[Definition~III.3.1]{Sh:e}. Readers familiar with \cite{Sh:e} 
may want to know that we fix  $\kappa=\aleph_1$ 
and $\Delta=\{\phi,\psi\}$  where $\psi(\vec x,\vec y)$ stands for $\phi(\vec y, \vec x)$. 
  
\begin{definition}\label{Def.wsl}
A $\phi$-chain $\cC$ is \emph{weakly $(\aleph_1,\phi)$-skeleton like} inside $A$ if for every 
$\vec a\in A^n$ there is a countable $\cC_{\vec a}\subseteq \cC$ such 
that for all $\vec b\preceq_\phi \vec c$ in $\cC$ 
with  $[\vec b,\vec c]_\phi$  disjoint from $\cC_{\vec a}$ we have
\[
A\models \phi(\vec b,\vec a)\liff \phi(\vec c,\vec a)
\]
 and 
 \[
A\models \phi(\vec a,\vec b)\liff \phi (\vec a,\vec c).
\]
\end{definition} 

\begin{remark} \label{R.I.-1} One can weaken the definition of 
weakly $(\aleph_1,\phi)$-skeleton like by requiring only that (with $\bar a$, $\cC_{\bar a}$, 
$\bar b$ and $\bar c$ as in Definition~\ref{Def.wsl}) 
\[
\bar a\leq_\phi \bar b\text{ if and only if } 
\bar a\leq_\phi \bar c
\]
and 
\[
\bar b\leq_\phi \bar a\text{ if and only if } 
\bar c\leq_\phi \bar a.
\]
All the statements about the notion of being weakly $(\aleph_1,\phi)$-skeleton 
like, except Lemma~\ref{L.2.6}, remain true for the modified notion. 
As a matter of fact, it is transparent that even their proofs remain unchanged.  
\end{remark} 

\begin{remark} \label{R.I.0} 
For $\vec a\in A^k$ and $\vec b\in A^n$ 
define
\[
\tp_\phi(\vec a/\vec b)=\{\psi(\vec x,\vec b)\colon A\models \psi(\vec a,\vec b)\}. 
\]
One may now consider a stronger indiscernibility requirement on a $\phi$-chain~$\cC$ 
than 
being weakly $(\aleph_1,\phi)$-skeleton like, defined as follows. 
\begin{enumerate}
\item [(*)] For every $k\in \bbN$ and $\vec a\in A^k$ there is a 
countable $\cC_{\vec a}\subseteq \cC$ such that for all $\vec b\preceq_\phi\vec c$ in $\cC$
with $[\vec b, \vec c]_\phi\cap \cC=\emptyset$ we have that 
\[
\tp_\phi(\vec a/\vec b)=\tp_\phi(\vec a/\vec c). 
\]
\end{enumerate}
The proofs of Theorem~\ref{T1} and Theorem~\ref{T1.m} can be easily modified 
to provide an ultrafilter $\cU$ such that  for a given linear order $I$ the ultrapower
$\prod_{\cU} A$ includes a $\phi$-chain $\cC$ isomorphic to $I$ and satisfying (*). 
See Remark~\ref{R.I.1} and Remark~\ref{R.I.2}. 
\end{remark}

The nontrivial part of the following is a special case of \cite[Claim~III.3.15]{Sh:e} that will be 
needed in~\S\ref{S.Defining}.

\begin{lemma} \label{L.triv} \label{L.wf}
Assume $\cC$ 
is a $\phi$-chain that 
is weakly $(\aleph_1,\phi)$-skeleton like in~$A$. 
Then $\cC^*$ is weakly $(\aleph_1,\phi)$-skeleton like inside $A$, 
and every interval of $\cC$ is 
weakly $(\aleph_1,\phi)$-skeleton like inside $A$. 
If $\cE\subseteq \cC$ is well-ordered (or conversely well-ordered) by $\preceq_\phi$
then $\cE$ is weakly $(\aleph_1,\phi)$-skeleton like in $A$. 
\end{lemma} 

\begin{proof} Only the last sentence requires a proof. 
For $\vec b\in A^n$ define $\cE_{\vec b}\subseteq \cE$ as follows. 
\begin{align*}
\cE_{\vec b}=
\{\min(\cE\cap [\vec c,\infty)_\phi)\colon \vec c\in\cC_{\vec b}\}
\cup 
\{\max(\cE\cap (-\infty,\vec c]_\phi)\colon \vec c\in \cC_{\vec b}\}. 
\end{align*}
Of course, for $\vec c\in \cC_{\vec b}$ the maximum as in the second set definition 
need not  exist. 
Each $\cE_{\vec b}$ is countable since 
every $\vec c\in \cC_{\vec b}$ produces at most two elements of~$\cE_{\vec b}$. 
For $\vec a\preceq_\phi \vec c$ in $\cE$ such 
that $[\vec a,\vec c]_\phi\cap\cE_{\vec b} =\emptyset$ we have 
that $[\vec a,\vec c]_\phi\cap\cC_{\vec b} =\emptyset$ and therefore 
 $\tp_\phi(\vec a/\vec b)=\tp_\phi(\vec c/\vec b)$. 
\end{proof}

If $\cC$ and $\cE$ are $\preceq_\phi$-chains in $A$ then we say $\cC$ and $\cE$ are
\emph{mutually cofinal} if 
 for every $\vec a\in \cC$ we have $\vec a\prec_\phi \vec b$ for all large 
 enough $\vec b\in \cE$
and for every $\vec b\in \cE$ we have  $\vec b\prec_\phi \vec a$ for 
all large enough $\vec a\in \cC$. 

\begin{lemma} \label{L0} Assume $\cC$ and $\cE$ are mutually cofinal 
$\phi$-chains in $A$. Then $\cf(\cC)=\cf(\cE)$. 
\end{lemma} 

Of course this is standard but since $\prec_\phi$ is not assumed to be a partial 
ordering on $A$ we shall prove it. Also note that if the condition 
`for every $\vec a\in \cC$ we have $\vec a\prec_\phi \vec b$ 
for all large enough $\vec b\in \cE$' 
is replaced by `for every $\vec a\in \cC$ we have $\vec a\prec_\phi \vec b$ 
for some $\vec b\in \cE$' and
the condition 
`for every $\vec b\in \cE$ we have  $\vec b\prec_\phi \vec a$ 
for all large enough $\vec a\in \cC$'
is replaced by 
is replaced by `for every $\vec b\in \cE$ we 
have $\vec b\prec_\phi \vec a$ for some $\vec a\in \cC$' 
then we cannot conclude $\cf(\cC)=\cf(\cE)$ in general.

\begin{proof}[Proof of Lemma~\ref{L0}]  
Assume $\kappa=\cf(\cC)<\cf(\cE)=\lambda$ and fix a cofinal $X\subseteq \cC$ of cardinality 
$\kappa$. For each $\vec a\in X$ pick $f(\vec a)\in \cE$ such that $\vec a\prec_\phi \vec b$
for all $\vec b$ such that $f(\vec a)\preceq_\phi \vec b$. 
The set $\{f(\vec a)\colon a\in X\}$ is not cofinal in $\cE$
and we can pick $\vec b\in \cE$ such that $f(\vec a)\preceq_\phi \vec b$ for all $\vec a\in X$. 
Now let $\vec a\in \cC$ be such that for all $\vec c\in \cC$ such 
that $\vec a\prec_\phi \vec c$ we have $\vec b\prec_\phi \vec c$. 
But there is    $\vec c\in X$ is such that $\vec a\prec_\phi \vec c$, 
and this is a contradiction. 
\end{proof} 

The following is  \cite[Lemma~III.3.7]{Sh:e} in the case $\kappa=\aleph_1$. 
We reproduce the proof for the convenience of the reader. 

\begin{lemma} \label{L1} 
Assume 
$\cC_0$, $\cC_1$ are  increasing,
  weakly $(\aleph_1,\phi)$-skeleton like, 
 $\phi$-chains in $A$. Also assume these two chains are mutually cofinal 
 and $m$ is such that both $\inv^m(\cC_0)$ and $\inv^m(\cC_1)$ are defined.  
 Then $\inv^m(\cC_0)=\inv^m(\cC_1)$. 
\end{lemma}

\begin{proof} 
The proof is by induction on $m$. If $m=0$ then this is Lemma~\ref{L0}. 
Now assume the assertion has been proved for $m$ and all pairs $\cC_0$ and $\cC_1$. 
Fix $\cC_0,\cC_1$ satisfying the assumptions for $m+1$ in place of $m$
 and let $\kappa=\cf(\cC_0)=\cf(\cC_1)$. 
Since $\inv^m(\cC_0)$ is defined, $\kappa\geq \aleph_1$. 
Since $\inv^{m+1}(\cC_0)$ is defined, $\calD(\kappa,\aleph_1)$ is a proper ideal
and $\kappa\geq \aleph_2$. 
  
For  an elementary sumbodel $N$ of $(A,\cC_0,\cC_1)$
consider
\begin{align*} 
\cC^0_N&=\{\vec b\in \cC_0\colon 
A\models \vec c\preceq_\phi \vec b\text{ for all }\vec c\in N^n\cap \cC_0\},\text{ and }\\
\cC^1_N&=\{\vec b\in \cC_1\colon 
A\models \vec c\preceq_\phi \vec b\text{ for all }c\in N^n\cap \cC_1\}. 
\end{align*} 
By our assumption that 
 $\inv^{m+1}(\cC_0)$ and $\inv^{m+1}(\cC_1)$ are defined  we have that for any 
 regular $\mu<\kappa$ the set of  
 $N\prec (A,\cC_0,\cC_1)$ of cardinality $\mu$
such that $\cf(N^n\cap \cC_0)\geq\aleph_1$
implies $\inv^m(\cC^0_N,\preceq_{\lnot\phi})$ is defined 
includes a club. 
 In particular, for club many  $N$ of size $\mu$ such 
 that $\cf(N^n\cap\cC_0)\geq \aleph_1$ 
 we have $\cf(\cC^0_N,\preceq_{\lnot\phi})\geq\aleph_1$. 
 Similarly, for club many $N$ of size $\mu$ such that 
 $\cf(N^n\cap\cC_1)\geq \aleph_1$ 
 we have that $\inv^m(\cC^1_N)$ is defined and 
 $\cf(\cC^1_N,\preceq_{\lnot\phi})\geq\aleph_1$. 

Now pick $N\prec A$ such that $\cf(N^n\cap \cC_0)$, $\cf(N^n\cap \cC_1)$, 
$\cf(\cC^0_N,\preceq_{\lnot\phi})$ and $\cf(\cC^1_N,\preceq_{\lnot\phi})$ are all uncountable
and $\inv^m(\cC^0_N,\preceq_{\lnot\phi})$ and $\inv^m(\cC^1_N,\preceq_{\lnot\phi})$ are defined. 
We shall prove that in this case $(\cC^0_N,\preceq_{\lnot\phi})$ and
$(\cC^1_N,\preceq_{\lnot\phi})$ are mutually cofinal. 

By the elementarity 
 $N^n\cap \cC_0$ and $N^n\cap \cC_1$ satisfy the assumptions of Lemma~\ref{L0}, 
and in particular $\cf(N^n\cap \cC_0)=\cf(N^n\cap \cC_1)$. 
Pick $\vec a\in \cC^0_N$. 
 Since $N^n\cap \cC_1$ and $N^n\cap \cC_0$ are mutually 
cofinal, by elementarity for all $\vec c\in N^n\cap \cC_1$ 
we have that $ \vec c\preceq_\phi \vec a$. 

Let $\cE_{\vec a}\subseteq \cC_1$ be a countable set such that 
for all $\vec b$ and $\vec c$ in $\cC_1$ satisfying
 $\vec b\preceq_\phi \vec c$ and $[\vec b,\vec c]_\phi\cap \cE_{\vec a}=\emptyset$
we have that
$A\models \phi(\vec b,\vec a)\liff \phi(\vec c,\vec a)$
 and 
$A\models \phi(\vec a,\vec b)\liff \phi (\vec a,\vec c)$. 
Since $\cE_{\vec a}$ is countable, by our assumptions on the cofinalities 
of $N^n\cap \cC_1$ and $(\cC^1_N,\preceq_{\lnot\phi})$
 for $\preceq_\phi$ cofinally many $\vec c\in N^n\cap \cC_1$ and
for $\preceq_{\lnot\phi}$-cofinally many $\vec d\in \cC^1_N$ we have
\[
A\models \vec c\preceq_\phi \vec a\liff \vec d\preceq_\phi \vec a. 
\]
Therefore for $\preceq_{\lnot\phi}$-cofinally many $\vec d\in \cC^1_N$ 
we have $\vec d\preceq_\phi \vec a$, 
i.e., $\vec a\preceq_{\lnot\phi} \vec d$. 

An analogous proof shows that for every $\vec e\in \cC^1$ 
and $\preceq_{\lnot\phi}$-cofinally many  $\vec d\in \cC^0$
 we have $\vec e\preceq_{\lnot\phi} \vec d$. 
  We have therefore proved that the $\phi$-chains $(\cC^0_N,\preceq_{\lnot\phi})$ 
and $(\cC^1_N,\preceq_{\lnot\phi})$ are mutually 
cofinal. They are both obviously weakly $(\aleph_1,\phi)$-skeleton like, 
and by the inductive hypothesis in this case we have 
$\inv^m(\cC^0_N,\preceq_{\lnot\phi})=\inv^m(\cC^1_N,\preceq_{\lnot\phi})$ if both of
these invariants are defined. 

 By the inductive hypothesis we have  
$\inv^{m+1}(\cC_0)=\inv^{m+1}(\cC_1)$. 
\end{proof}


%
%

\subsection{Defining an invariant over a submodel} 
\label{S.Defining}

Assume $Z\prec A^n$. 
By $\tp_\phi(\vec a/X)$ we denote the $\phi$-type of $\vec a\in A^n$ in 
the signature $\{\phi\}$ over $Z$, or in symbols
\[
\tp_\phi(\vec a/Z)=
\{\phi(\vec x, \vec b)\colon \vec b\in Z, A\models \phi(\vec a, \vec b)\}
\cup 
\{\phi(\vec b, \vec x)\colon \vec b\in Z, A\models \phi(\vec b, \vec a)\}.
\]
If $B\subseteq A$ (in particular, if $B$ is an elementary submodel of $A$) we shall write $\tp_
\phi(\vec a/B)$
for $\tp_\phi(\vec a/B^n)$. 
 Write $\tp_\phi(\vec a/\vec e)$ for $\tp_\phi(\vec a/\{\vec e\})$. 
 
\begin{lemma} \label{L.2.6} A $\phi$-chain 
 $\cC$ in $A$  is  weakly $(\aleph_1,\phi)$-skeleton like in $A$
 if and only if for every   $\vec a\in A^n$
 there exists a countable $\cC_{\vec a}\subseteq \cC$ with the 
property that for $\vec c$ and $\vec d$  in $\cC$  the condition  
\[
\cC_{\vec a}\cap (-\infty, \vec c]_\phi=\cC_{\vec a}\cap (-\infty,\vec d]_\phi
\]
implies $\tp_\phi(\vec a/\vec c)=\tp_\phi(\vec a/\vec d)$. 
\end{lemma} 

\begin{proof} Immediate from  Definition~\ref{Def.wsl}. 
\end{proof}

\begin{definition} \label{Def.Defines} 
Assume $B$ is an elementary submodel of  $A$,  $m\in \bbN$, and $\bd$ is an $m$-invariant. 
We say that $ \vec c\in A^n\setminus B^n$ \emph{defines} an $(A,B,\phi,m)$-invariant $\bd$
if there are 
\begin{enumerate}
\item\label{Def.Defines.1}  (nonempty) linear orders $J$ and $I$,  and 
\item  \label{Def.Defines.2}
$\vec a_j\in B^n$ for $j\in J$  and $\vec a_i\in A^n\setminus B^n$ for $i\in I$, 
\pushcounter
\end{enumerate}
such that 
\begin{enumerate}
\popcounter
\item \label{Def.Defines.3}$\langle \vec a_i\colon i\in J+I^*\rangle$ is 
a $\phi$-chain in $A$ that is weakly $(\aleph_1,\phi)$-skeleton like in $A$, 
\item \label{Def.Defines.4}$\tp_\phi(\vec a_i/B)=\tp_\phi(\vec c/B)$  for all $i\in I$,
\item \label{Def.Defines.5} $\bd=\inv^m(I)$, and
\item \label{Def.Defines.6} 
if $J', I'$,  $\vec a_i'$ for $i\in J'\cup I'$ and $\bd'$  
satisfy conditions \eqref{Def.Defines.1}--\eqref{Def.Defines.5} then 
$\inv^m(\bd')=\inv^m(\bd)$. 
\end{enumerate}
Let $\INV^{m}(A,B,\phi)$ denote the set of all $m$-invariants $\bd$ such that some~$\vec c$ 
defines an $(A,B,\phi,m)$-invariant $\bd$. 
\end{definition}

Conditions \eqref{Def.Defines.1}--\eqref{Def.Defines.5} of Definition~\ref{Def.Defines} 
imply \eqref{Def.Defines.6} of Definition~\ref{Def.Defines}.  
This is a consequence of  Lemma~\ref{L.tp} and the fact that cofinalities
occurring in   invariants that are defined in the sense of \S\ref{S.m} or \S\ref{S.m,lambda}  
are uncountable.

The following notation will be useful. 
Assume $\cC$ is a $\phi$-chain that is weakly $(\aleph_1,\phi)$-skeleton like in $A$
and $B$ is an elementary submodel of $A$. 
For $\vec c\in \cC\setminus B^n$ let 
\[
\cC[B,\vec c]=\{\vec a\in \cC\colon (\forall \vec b\in B^n\cap\cC) \vec c\preceq_\phi \vec b
\liff \vec a\preceq_\phi\vec b \}. 
\]
{\bf We shall always consider $\cC[B,\vec c]$ with respect to 
the reverse order,~$\preceq_{\lnot\phi}$.}


\begin{lemma} \label{L.tp-phi} 
Assume $\cC=\langle a_i\colon i\in I\rangle$ is a $\phi$-chain 
that is weakly $(\aleph_1,\phi)$-skeleton like in $A$. 
Assume $B$ is an elementary submodel of  $A$ and $\vec c\in \cC\setminus B^n$ are such that 
 \begin{enumerate}
 \item  $\cC_{\bar b}\cap \cC[B,\bar c]\cap (-\infty,\bar c]_\phi=\emptyset$ for all $\bar b\in B^n$, and
 \item   $\bd=\inv^m(\cC[B,\vec c],\preceq_{\lnot\phi})$ is well-defined. 
\pushcounter
\end{enumerate}
Then $\vec c$ defines the $(A,B,\phi,m)$-invariant $\bd$.   
\end{lemma} 

\begin{proof} Let $J_0$ be a well-ordered 
$\preceq_\phi$-cofinal subset of 
\[
\{i\in I\colon \vec a_i\in B^n\text{ and } \vec a_i\preceq_\phi \vec c\}
\]
of minimal order type. 
By Lemma~\ref{L.wf} the $\phi$-chain $\langle a_i\colon i\in J_0\rangle$
is weakly $(\aleph_1,\phi)$-skeleton like in~$A$. Let $I_0= 
  \{i\in I\colon \vec a_i\in \cC[B,\vec c]\}$. 
    We need to check that  $I_0,J_0$ and $\langle\vec a_i\colon 
  i\in J_0+I_0^*\rangle$ satisfy  \eqref{Def.Defines.1}--\eqref{Def.Defines.6}
     of
  Definition~\ref{Def.Defines}. 

   Clauses \eqref{Def.Defines.1}--\eqref{Def.Defines.2} 
  are immediate.
  As an interval of a
  weakly $(\aleph_1,\phi)$-skeleton like order, $\langle a_i\colon i\in I_0\rangle$
  is weakly $(\aleph_1,\phi)$-skeleton like. 
  Therefore  clauses \eqref{Def.Defines.3} 
  follows. 
In order to prove~\eqref{Def.Defines.4} pick $\bar b\in B^n$ 
and $\bar d\in \cC[B,\bar c]\cap (-\infty,\bar c]_\phi$. Then 
$[\bar d,\bar c]_\phi\cap \cC_{\bar b}=\emptyset$, 
hence $\tp_\phi(\bar c/\bar b)=\tp_\phi(\bar d/\bar b)$. Sine $\bar b\in B^n$
was arbitrary, we have $\tp_\phi(\bar c/B)=\tp_\phi(\bar d/B)$ and we have proved 
\eqref{Def.Defines.4}. 
Clause~\eqref{Def.Defines.5} is automatic, 
  and \eqref{Def.Defines.6} follows by Lemma~\ref{L.tp} below.  
\end{proof}

\begin{lemma} \label{L.tp} Assume $I_0,I_1, J_0,J_1$ are linear orders and 
$\langle \vec a_i\colon i\in J_0+I_0^*\rangle$ and 
$\langle \vec b_i\colon i\in J_1+I_1^*\rangle$ are weakly $(\aleph_1,\phi)$-skeleton 
like $\phi$-chains in $A$ such that 
\begin{enumerate}
\popcounter
\item $\vec a_i\in B^n$ if and only if $i\in J_0$ and $\vec b_i\in B^n$ if and only if $i\in J_1$, 
\item $\tp_\phi(\vec a_i/B)=\tp_\phi(\vec b_j/B)$ for all $i\in I_0$ and all $j\in I_1$, 
\item each of $\cf(I_0)$, $\cf(I_1)$, $\cf(J_0)$, and $\cf(J_1)$ is uncountable.  
\end{enumerate}
If $\inv^m(I_0)$ and $\inv^m(I_1)$ are both defined then $\inv^m(I_0)=\inv^m(I_1)$. 
\end{lemma} 

\begin{proof} Pick $i(0)\in I_0$. Since $\tp_\phi(\vec a_{i(0)}/B)=\tp_\phi(\vec b_j/B)$
for some (any) $j\in I_1$, we have that $\vec b_i\preceq_\phi \vec a_{i(0)}$ for all $i\in J_1$. 
Since $\cf(J_1)$ and $\cf(I_1)$ are both uncountable and 
since $\langle b_i\colon i\in J_1+I_1^*\rangle$ is weakly $(\aleph_1,\phi)$-skeleton like, we 
conclude that for large enough $i\in I_1$ we have $\vec a_{i(0)}\preceq_{\lnot\phi} \vec b_i$. 

The analogous argument shows that for every $i(1)\in I_1$ and all large enough 
$i\in I_0$ we have $\vec a_{i(1)}\preceq_{\lnot\phi} \vec b_i$. 
Then $\langle \vec a_i\colon i\in I_0\rangle$ and $\langle \vec b_i\colon i\in I_1\rangle$ are, 
when ordered by $\preceq_{\lnot\phi}$, mutually cofinal. 

By Lemma~\ref{L1} we have that $\inv^m(I_0)=\inv^m(I_1)$ if both of these invariants
are defined, and   
the claim follows. 
\end{proof}



\subsection{Representing invariants} \label{S.RI} 
In addition to  $A$, $\phi$ and $m$  fixed in \S\ref{S.OP} we 
distinguish $\lambda=|A|$. 
A \emph{representation} of $A$ is a continuous chain of elementary submodels $A_\xi$, for 
$\xi<\cf(\lambda)$, 
of $A$
such that $|A_\xi|<|A|$ for all $\xi$ 
and $\bigcup_{\xi<\cf(\lambda)}A_\xi=A$. 
  
Define a set $\INV^{m,\lambda}(A,\phi)$  of $m,\lambda$-invariants (see \S\ref{S.m,lambda})
by cases  as follows. 
Whenever~$\bd$ is an $m$-invariant, or an $m,\lambda$-invariant, for $m\geq 1$ we write 
$\langle \bd_\xi\colon \xi<\cf(\bd)\rangle$ for its representation. Although  this representation
is not unique, it is unique modulo the appropriate filter
$\calD(\cf(\lambda),\aleph_1)$ or $\calD_{h_\lambda}(\aleph_2)$.

\subsubsection{Assume $\lambda$ is regular} 
\label{S.m,lambda.1} 
Then $\bd$ is an \emph{$m,\lambda$-invariant}
of $A,\phi$ if $\bd$ is an  $m,\lambda$-invariant and for every 
representation $A_\xi$, $\xi<\lambda$ of $A$ we have 
\[
\{\xi\colon \bd_\xi\in \INV^{m}(A,A_\xi,\phi)\}\in \calD(\lambda,\aleph_1). 
\]

\subsubsection{Assume $\lambda$ is singular and $\cf(\lambda)>\aleph_1$.} 
\label{S.m,lambda.2} 
Then $\bd$ is an \emph{$m,\lambda$-invariant}
of $A,\phi$ if $\bd$ is an  $m,\lambda$-invariant and for every 
representation $A_\xi$, $\xi<\cf(\lambda)$ of $A$ we have 
\[
\{\xi\colon \bd_\xi\in \INV^{m}(A,A_\xi,\phi)\}\in \calD(\cf(\lambda),\aleph_1). 
\]

\subsubsection{Assume $\lambda$ is singular and $\aleph_1\geq\cf(\lambda)$}
\label{S.m,lambda.3} 
Fix $h\colon \aleph_2\to \cf(\lambda)$ as in Lemma~\ref{L.h}. Then  $\bd$ is an  
 \emph{$m,\lambda$-invariant} of $A,\phi$ if $\bd$ is an $m,\lambda$-invariant and for every 
representation $A=\bigcup_{\xi<\cf(\lambda)}A_\xi$ there is $\xi<\cf(\lambda)$ such that
\[
\{i<\aleph_2\colon \bd_i\in \INV^{m}(A,A_\xi,\phi)\text{ and } 
h(i)>|A_\xi|\}\in \calD_h(\aleph_2). 
\]

\begin{lemma} \label{L.J}
Assume $A,\phi,m$ and $\lambda=|A|$ are as above. 
Also assume $\cC=\langle \vec a_j\colon j\in J\rangle$ is a $\phi$-chain in $A$ that is weakly
$(\aleph_1,\phi)$-skeleton like in $A$. 
If $\inv^{m,\lambda}(J)$ is defined then $\inv^{m,\lambda}(J)\in \INV^{m,\lambda}(A)$. 
\end{lemma} 

\begin{proof} This is really three lemmas wrapped up in one. We prove each of the 
three cases, depending on the cofinality of $\lambda$ (\S\ref{S.m,lambda.1},
\S\ref{S.m,lambda.2} and 
\S\ref{S.m,lambda.3}) 
separately. 

\subsubsection{Assume $\lambda$ is regular} 
 Fix a representation $A_\xi$, $\xi<\lambda$, of $A$. 
Let $\bfC\subseteq \lambda$ be the club consisting of all $\xi$ 
such that for every $\bar a\in A_\xi^n$ we have $\cC_{\bar a}\subseteq A_\xi^n$. 
By the assumption $\cf(J)=\lambda$ and
we may clearly assume $m\geq 1$. 
Let 
\[
\bd=\langle \bd_\xi\colon \xi<\lambda\rangle/\calD(\lambda,\aleph_1).
\] 
Fix $\xi\in \bfC$ such that $\cf(\xi)=\cf(\cC\cap A_\xi^n)\geq \aleph_1$
and $\bd_\xi$ is defined. Since $\cf(J)=\lambda$ by 
\S\ref{S.m,lambda.1} 
the set of such $\xi$ belongs to $\calD(\lambda,\aleph_1)$. 
It will therefore suffice to show that for every such $\xi$ some $\vec c$ defines
 defines the $(A,A_\xi,\phi,m)$-invariant $\bd_\xi$.  

Pick $\vec c\in\cC$ such that $(-\infty,\vec c]_\phi\cap A_\xi^n\supseteq
\cC\cap A_\xi^n$. 
Let $I^\xi$ be the order with the underlying set 
$\{i\in J\colon \vec a_i\in \cC[A_\xi,\vec c]\}$, so that $\inv^m(I^\xi)=\bd_\xi$. 
 Then 
 \[
 \cf(\cC\cap A_\xi^n)=\cf(\xi)\geq\aleph_1
 \]
 and 
 \[
 \cf(\cC[A_\xi,\vec c],\preceq_{\lnot\phi})=\cf(\bd_\xi)\geq\aleph_1. 
 \]
Since $\bar a\in A_\xi^n$ implies $\cC_{\bar a}\subseteq A_\xi^n$, 
 Lemma~\ref{L.tp-phi} implies that 
$\vec c$ defines the $(A,A_\xi,\phi,m)$-invariant $\bd_\xi$.  

\subsubsection{Assume $\lambda$ is singular and $\aleph_1<\cf(\lambda)$}
Fix a representation $A_\xi$, $\xi<\cf(\lambda)$, of $A$. 
By the assumption $\cf(J)=\cf(\lambda)$ and
we may clearly assume $m\geq 1$. 
Pick $\xi(0)<\cf(\lambda)$ such that $A_{\xi(0)}\cap \cC$ is cofinal in $\cC$. 

Let $\bd=\langle \bd_\xi\colon \xi<\cf(\lambda)\rangle/\calD(\cf(\lambda),\aleph_1)$. 
Hence $J=\sum_{\xi<\cf(\lambda)} J_\xi^*$ with $\inv^{m-1}(J_\xi)=\bd_\xi$
for $\calD(\cf(\lambda),\aleph_1)$-many $\xi$. 
By \S\ref{S.m,lambda.2} we also have
 $\cf(\bd_\xi)=\cf(J_\xi)>|A_\xi|$ for $\calD(\cf(\lambda),\aleph_1)$ many $\xi$. 
 It will therefore suffice to show that for every such $\xi$ some $\vec c$ defines
  the $(A,A_\xi,\phi,m)$-invariant $\bd_\xi$.  

Since $\cf(J_\xi)>|A_\xi|$, 
for such $\xi$ we can pick $j(0)\in J_\xi$ 
such that 
\[
\{ \vec a_i\colon i\in J_\xi, i>j(0)\}
\cap (A_\xi^n\cup \bigcup\{\cC_{\bar a}\colon \bar a\in A_\xi^n\})=\emptyset.
\] 
Let $\vec c=\vec a_{j(0)}$. 
Then 
\[
\cf(A_\xi^n\cap \cC\cap (-\infty,\vec c]_\phi,\preceq_\phi)=\cf(\xi)\geq\aleph_1
\]
and
\[
\cf(\cC[A_\xi,\vec c])=\cf(\bd_\xi)\geq\aleph_1.
\]
By Lemma~\ref{L.tp-phi} we have that $\vec c$ defines
the $(A,A_{\xi(0)},\phi,m)$-invariant $\bd_\eta$.

\subsubsection{Assume $\lambda$ is singular and $\cf(\lambda)\leq\aleph_1$}
 Fix a representation $A_\xi$, $\xi<\cf(\lambda)$, of $A$. 
By the assumption $\cf(J)=\aleph_2$ and
we may clearly assume $m\geq 1$. 
Let $\bd=\langle \bd_\xi\colon \xi<\aleph_2\rangle/\calD_{h_\lambda}(\aleph_2)$
and write $J=\sum_{\zeta<\aleph_2} J_\zeta^*$ so 
that $\inv^{m-1}(J_\zeta)=\inv^{m-1}(\bd_\zeta)$ for $\calD(\aleph_1,h_\lambda)$-many $\zeta$.

Fix $\xi(0)<\cf(\lambda)$ such that $A_{\xi(0)}\cap \cC$ is cofinal in $\cC$. 
The set of $\eta<\aleph_2$ such that $h(\eta)>\xi(0)$ and $\cf(\bd_\eta)>|A_{\xi(0)}|$ 
belongs to $\calD_h(\aleph_2)$, and it will suffice to show that 
for such $\eta$ some $\vec c$ 
 defines the $(A,A_{\xi(0)},\phi,m)$-invariant $\bd_\eta$.  

Since $\cf(\bd_\eta)=\cf(J_\eta)>|A_{\xi(0)}|$, we can pick $j(0)\in J_\eta$
such that 
\[
\{ \vec a_i\colon i\in J_\eta, i>j(0)\}
\cap  (A_{\xi(0)}^n\cup\bigcup\{\cC_{\bar a}\colon \bar a\in A_\xi^n\})=\emptyset. 
\] 
Let $\vec c=\vec a_{j(0)}$. Then 
\[
\cf(A_{\xi(0)}^n\cap \cC\cap (-\infty,\vec c]_\phi,\preceq_\phi)=\cf(\eta)\geq\aleph_1
\]
and 
\[
\cf(\cC[A_{\xi(0)}, \vec c],\preceq_{\lnot\phi})=\cf(\bd_\eta)\geq \aleph_1. 
\]
By Lemma~\ref{L.tp-phi} we have that $\vec c$ defines
the $(A,A_{\xi(0)},\phi,m)$-invariant $\bd_\eta$. 

This exhausts the cases and concludes the proof of Lemma. 
\end{proof} 

\subsection{Counting the number of invariants of a model} 
\label{S.counting} 
We would like to prove  the inequality $|\INV^{m,\lambda}(A,\phi)|\leq |A|$ 
for every model $A$ of cardinality $\geq\aleph_2$. 
Instead we prove a sufficiently strong
  approximation 
 to this inequaity. As a courtesy to the reader we start by 
 isolating the following triviality. 

\begin{lemma}\label{L.trivial} For every cardinal $\lambda$ and every 
$\cX\subseteq \cP(\lambda)$
of cardinality $>\lambda$ there is $\xi<\lambda$ such that $|\{x\in \cX\colon \xi\in x\}|>\lambda$. 
\end{lemma} 

\begin{proof} We may assume $|\cX|=\lambda^+$ and enumerate $\cX$ as 
$\{x_\eta\colon \eta<\lambda^+\}$. If the conclusion of lemma fails then 
$f(\xi)=\sup\{\eta<\lambda^+\colon \xi\in x_\eta\}$ defines a cofinal function 
from $\lambda$ to $\lambda^+$. 
\end{proof}

See the paragraph before Lemma~\ref{L.many-invariants}  for the definition of 
disjoint representing sequences. 

\begin{lemma} \label{L.disj.1} For $A,\phi,m$ as usual and $\lambda=|A|$ 
every  set of disjoint representing sequences of invariants in 
$\INV^{m,\lambda}(A,\phi)$ has size at most $\lambda$. 
\end{lemma} 

\begin{proof} Let us prove the case when $\lambda$ is regular. 
We may assume $m\geq 1$ since the case $m=0$ is trivial. 
Assume the contrary and let $\bd(\eta)$, for $\eta<\lambda^+$, 
be disjoint representing sequences of elements of $\INV^{m,\lambda}(A,\phi)$. 
Let $\bd(\eta)=\langle \bd(\eta)_\xi\colon \xi<\lambda\rangle/\calD(\lambda,\aleph_1)$. 
Fix a representation $A_\xi$, for $\xi<\lambda$, of $A$.

For each $\eta<\lambda^+$ fix 
$S_\eta\in \calD(\lambda,\aleph_1)$ such that for every $\xi\in S_\eta$ some
 $\vec c_\xi$ defines an $(A,A_\xi,\phi,m)$-invariant $\bd(\eta)_\xi$. 
By Lemma~\ref{L.trivial} there is $\xi<\lambda$ such that
 $\lambda^+$ distinct $(A,A_\xi,\phi,m)$-invariants are defined by elements of $A^n$. 
 Since $|A|=\lambda$, this is impossible. 

The proofs of the two cases when $\lambda$ is singular  are   almost identical to the above
proof  and are therefore omitted. 
 \end{proof} 
 
\begin{prop} \label{P.counting} 
Assume $\lambda\geq \aleph_2$ and 
 $\bbK$ is a class of models of cardinality~$\lambda$.  
If there are $n$ and a $2n$-ary formula $\phi$   such that for every linear 
order~$I$ of cardinality $\lambda$ there exists a model $A\in \bbK$ 
 such that $I$ is isomorphic to a weakly $(\aleph_1,\phi)$-skeleton like
 $\phi$-chain in $A^n$, then there are $2^\lambda$ nonisomorphic models in $\bbK$. 
 \end{prop} 
  
  \begin{proof} Let $I$ be a linear order and let $A$ be a model such that 
  $I$ is isomorphic to a weakly $(\aleph_1,\phi)$-skeleton like $\phi$-chain in $A$. 
  By Lemma~\ref{L.J}, $\inv^{m,\lambda}(I)\in \INV^{m,\lambda}(A)$ and 
by Lemma~\ref{L.disj.1}, 
$\INV^m(A)$ has cardinality at most $\lambda$ for every 
$A\in\bbK$. By the pigeonhole principle there are $2^{\lambda}$ 
nonisomorphic ultraproducts  
elements of $\bbK$. 
\end{proof} 
  


\section{Construction of ultrafilters} 
\label{S.Construction} 


The main result of this section  is Proposition~\ref{P.construction} below. 
 Its version in  which 
$M_i=(\bbN,<)$ for all $i$ was proved in  \cite[Lemma~4.7]{KShTS:818}
and some of the ideas are taken from this proof. 
Recall that if $D$ is a filter on $\lambda$ then  $D^+$ is the coideal of all sets positive 
with respect to $D$, or in symbols
\[
D^+=\{X\subseteq \lambda\colon X\cap Y\neq \emptyset\text{ for all }Y\in D\}.
\]
If $D$ is a filter on $\lambda$ and $\cG\subseteq \bbN^\lambda$ then we say 
$\cG$ is \emph{independent mod $D$} if
for all $k\in\bbN$, all distinct  $g_0, \dots, g_{k-1}$ in $\cG$ and all 
$j_0,\dots, j_{k-1}$ in $\bbN$ the set
\[
\{\xi<\lambda\colon g_0(\xi)=j_0,\dots g_{k-1}(\xi)=j_{k-1}\}
\]
belongs to $D^+$. Note that it is not required that $j_i$ be distinct. 

Write $\FI(\cG)$ for the family of all finite partial functions $h$ from $\cG$ into $\bbN$. 
For $h\in \FI(\cG)$ write 
\[
A_h=\{n\in \bbN\colon \ff(n)=h(\ff)\text{ for all }\ff\in \dom(h)\}. 
\]
Let 
\[
\FI_s(\cG)=\{A_h\colon h\in \FI(\cG)\}. 
\]
Lemma~\ref{L.Cohen} below   a special case of \cite[Claim~VI.3.17(5)]{Sh:c}. 
We include   its proof for  convenience of the reader. 
We shall write $X\subseteq^D Y$ for $X\setminus Y=\emptyset$ mod $D$ and
$X=^D Y$ for $X\Delta Y=\emptyset$ mod $D$. 

\begin{lemma} \label{L.Cohen} 
Assume $D$ is a filter on $\lambda$ and $\cG\subseteq \bbN^\lambda$ is a family of functions 
independent mod $D$. Furthermore, assume $D$ is a maximal filter 
such that~$\cG$ is independent mod $D$. 
Then for every $X\subseteq \lambda$ there is a countable subset $\cA\subseteq \FI(\cG)$
such that 
\begin{enumerate}
\item For every $h\in \cA$ either $A_h\subseteq^D X$  or
$A_h\cap X=^D\emptyset$. 
\item For every $h'\in \FI(\cG)$ there is $h\in \cA$ such that $A_{h'}\cap A_{h}\neq^D \emptyset$.  
\end{enumerate}
\end{lemma}

\begin{proof} Let $\cA_0$ be the set of all $Y\in \FI_S(\cG)$ such that 
(1) holds. Assume for a moment that $\cA_0$ satisfies (2). 
Then let $\cA\subseteq \cA_0$ be maximal with respect to the property that
$A_{h}\cap A_{h'}=\emptyset$ mod $D$ for all $h\neq h'$ in $\cA$. 
Then $\cA$ still satisfies (1) and (2) and 
the standard $\Delta$-system argument (see \cite{Sh:c} or 
\cite{Ku:Book}) shows that $\cA$ is countable. 

We may therefore assume there is $h\in \FI(\cG)$ such that for all $s\in \cA_0$ we have
both $A_s\cap A_h\neq \emptyset$ mod $D$ and $A_h\setminus A_s\neq\emptyset$ mod 
$D$. Let $D'$ be the filter generated by $D$ and $X\cap A_h$. Since the first part of (1) fails 
for $h$, we have that 
$D'$ is a proper extension of $D$. Since the second part of (1) fails for every $s$ 
extending $h$, we have
that $\cG$ is independent modulo $D'$. 
This contradicts the assumed maximality of $D$. 
\end{proof}

Lemma~\ref{L.Cohen} implies that 
 for every $X\subseteq\bbN$ there is a countable $\cG_0\subseteq \cG$ such that
$\cA$ satisfying the above conditions is included in $\FI_s(\cG_0)$. In this situation we say 
$X$ is \emph{supported} by $\cG_0$.

\begin{prop}\label{P.construction} 
Assume $\phi(\vec x, \vec y)$ is  a  formula and $M_i$, for $i\in \bbN$, 
are models of the same signature such that in $M_i$ there is a $\preceq_\phi$-chain 
of length~$i$. 
Then for every linear order $I$ of cardinality $\leq \fc$ there exists an 
ultrafilter $\cU$ on $\bbN$ such that  $\prod_{\cU} M_n$ 
includes a weakly $(\aleph_1,\phi)$-skeleton like  
 $\phi$-chain~$\cC$ isomorphic to~$I$. 
\end{prop}

\begin{proof} In order to simplify the notation and release the bound variable $n$ 
we shall assume that 
$\phi$ is a binary formula and hence the elements 
of the $\phi$-chain~$\cC$ will be elements 
of $A$ instead of $n$-tuples of elements from $A$. 
Let $a_i(n)$, for $0\leq i<n$, be a $\preceq_\phi$-chain in $M_n$. 
For convenience of notation, we may assume 
\[
a_i(n)=i
\]
for all $i$ and $n$, and we also write $a_i(n)=n-1$ if $i\geq n$.  
Fix an independent family $\cG$ of size $\fc$ of functions $f\colon \bbN\to \bbN$
(see \cite[Appendix, Theorem~1.5(1)]{Sh:c}). 
Fix a filter $D$ on $\bbN$ such that $\cG$ is independent with respect to $D$ and 
$D$ is a maximal (under the inclusion) filter with this property. 
Let   $\FI(\ccG)$, $A_h$ for  $h\in \FI(\ccG)$ and $\FI_s(\ccG)$
be as  introduced before Lemma~\ref{L.Cohen}. 
The  following is an immediate consequence of  
 Lemma~\ref{L.Cohen} (i.e.,  of \cite[Claim~VI.3.17(5)]{Sh:c}). 
 
\begin{claim}\label{C.const.1} 
For every $g\in \prod_{n\in \bbN} M_n$ there is a countable set $\bbS_g\subseteq I$
such that  for all $l\in\bbN$ both sets
\begin{align*} 
X_{g,l}&=\{n\colon M_n\models \phi(a_l(n),g(n))\}\\
Y_{g,l}&=\{n\colon M_n\models \phi(g(n),a_l(n))\}
\end{align*} 
are supported by $\{f_i\colon i\in \bbS_g\}$. \qed 
\end{claim} 

Fix an enumeration of $\ccG$ by elements of $I$ and write $\ccG=\{\ff_i\colon i\in I\}$. 
For $i<j$ in $I$ write $[i,j]_I$ for the interval $\{k\in I\colon i\leq k\leq j\}$. 
For elements $a\preceq_\phi b$ in a model $M$ write 
\[
[a,b]_\phi=\{c\in M\colon a\preceq_\phi c \text{ and } c\preceq_\phi b\}. 
\]
Since $\preceq_\phi$ is not necessarily transitive, this notation should be taken 
with a grain of salt. 
For $i<j$ in $I$ write
\[
B_{ij}=\{n\colon \ff_i(n)\preceq_\phi \ff_j (n)\}. 
\]
(Note that by our convention about $a_i(n)$ we have that
$\ff_i(n)\preceq_\phi \ff_j(n)$ is equivalent to $\ff_i(n)\leq \ff_j(n)$.)
For $g\in \prod_{n\in\bbN} M_n$ and $i<j$ in $I$ such that $[i,j]_\phi\cap \bbS_g=\emptyset$ let
\begin{align*}
C_{gij}=\{n\colon&
M_n\models \phi(\ff_i(n), g(n))\liff \phi(\ff_j(n), g(n))\\
&\text{ and } 
M_n\models\phi(g(n),\ff_i(n))\liff \phi(g(n),\ff_j(n))\}.
\end{align*}
In other words, $C_{gij}=\{n\colon \tp_\phi(f_i(n)/g(n))=\tp_\phi(f_j(n)/g(n))$, 
with $\tp_\phi$ as computed in $M_n$.

\begin{claim} \label{C.fip} 
The family of all sets
$B_{ij}$ for $i<j$ in $I$ and $C_{gij}$ for $g\in \prod_{n\in\bbN} M_n$ and
$i<j$ in $I$ such that $[i,j]_i\cap \bbS_g=\emptyset$ has the finite intersection 
property. 
\end{claim} 

\begin{proof} It will suffice to show that for $\bar k\in\bbN$, $i(0)<\dots <i(\bar k-1)$ in $I$, and 
$g(0),\dots, g(\bar k-1)$ in $\prod_{n\in\bbN} M_n$ the set 
\begin{multline*}
\bigcap_{l<m<\bar k} B_{i(l),i(m)}\cap\\
\bigcap\{C_{g(k),i(l), i(m)}\colon k<\bar k, l<m<\bar k, \text{ and } 
[i(l),i(m)]_I\cap \bbS_{g(k)}=\emptyset\}
\end{multline*}
is nonempty. 
Let
\[
\bbS=\bigcup_{k<\bar k}\bbS_{g(k)}.
\]
Write $\cT=\{{i(k)}\colon k<\bar k\}$, also 
$\cT^{\cG}=\{f_i\colon i\in \cT\}$ and $\bbS^{\cG}=\{f_i\colon i\in \bbS\}$. 
 
 Pick~$h_m$, for $m\in\bbN$, in $\FI(\bbS^{\cG}\setminus \cT^{\cG})$ so that
\begin{enumerate}
\item $h_m\subseteq h_{m+1}$ for all $m$ and 
\item For all $h\in \FI(\cT^{\cG})$, all $l\in\bbN$ and
 all $k<\bar k$, for all but finitely many $m$ we have either
\begin{enumerate}
\item [(i$_X$)] $(\forall^D n\in A_{h_m\cup h})M_n\models \phi(a_l(n),g(k)(n))$, or
\item [(ii$_X$)] $(\forall^D n\in  A_{h_m\cup h})M_n\models \lnot\phi(a_l(n),g(k)(n))$
\end{enumerate}
and also either 
\begin{enumerate}
\item [(i$_Y$)] $(\forall^D n\in  A_{h_m\cup h})M_n\models \phi(g(k)(n),a_l(n))$, or
\item [(ii$_Y$)]  $(\forall^D n\in  A_{h_m\cup h})M_n\models \lnot\phi(g(k)(n),a_l(n))$. 
\end{enumerate}
\end{enumerate}
The construction of $h_m$ proceeds recursively as follows. 
Enumerate all triples $(h,k,l)$ in $\FI(\cT^{\cG})\times \bar k\times\bbN$ by 
elements of $\bbN$. Let $h_0=\emptyset$. If $h_m$ has been chosen and $(h,k,l)$ is 
the $m$-th triple then use the fact that $X_{g(k),l}$ and $Y_{g(k),l}$ are supported by $\bbS$ 
(Claim~\ref{C.const.1}) to find $h_{m+1}\in \FI(\bbS^{\cG}\setminus \cT^{\cG})$ such that 
$A_{h_{m+1}\cup h}$ satisfies one of (i$_X$) and (ii$_X$) and one of (i$_Y$) or (ii$_Y$). 
Then  the sequence of $h_m$ constructed as above clearly satisfies the requirements. 

In order to complete the proof we need to show that there exist $h\in \FI(\cT^{\cG})$ and $n$ such 
that 
\begin{multline}\label{Eq.1}  
A_{h_n\cup h}\subseteq ^D 
\bigcap_{l<m<\bar k} B_{i(l),i(m)}\cap\\
\bigcap\{C_{g(k),i(l), i(m)}\colon k<\bar k, l<m<\bar k, \text{ and } 
[i(l),i(m)]_I\cap \bbS_{g(k)}=\emptyset\}.
\end{multline}
In order to have $A_{h_n\cup h}\subseteq^D B_{i(l),i(m)}$ it is necessary and sufficient 
to have  $h(i(l))<h(i(m))$. We shall therefore consider only $h$ that are \emph{increasing} 
in this sense. An increasing function in $\FI(\cT^{\cG})$ is uniquely determined by its range. 
For $t\in [\bbN]^{\bar k}$ let $h_t$ denote the increasing function in $\FI(\cT^{\cG})$ 
whose range is equal to $t$.

Assume for a moment that for every $t\in [\bbN]^{\bar k}$ there are
$k,l,m$ such that   for all $n$ we have 
$A_{h_n\cup h_t} \not\subseteq^D  C_{g(k),i(l),i(m)}$ and therefore by the choice 
of the sequence  $\{h_n\}$ that 
\begin{equation*}  
A_{h_n\cup h_t} \cap C_{g(k),i(l),i(m)}=^D \emptyset. 
\end{equation*}
For $t\in [\bbN]^{\bar k}$ let $\psi(t)$ be the lexicographically minimal triple $(k,l,m)$ 
such that this holds for a large enough $n$. By Ramsey's theorem, there are an infinite 
$Z\subseteq \bbN$ and a triple $(k^*,l^*,m^*)$ such that for every $t\in [\bbN]^{\bar k}$ 
we have $A_{h_n\cup h_t} \cap C_{g(k),i(l),i(m)}=^D \emptyset$.

Let $N=|[i(l^*),i(m^*)]_I\cap \cT|$ and 
find $t\in [Z]^{\bar k}$ such that the set
\[
[h_t(i(l^*)),h_t(i(m^*))]\cap Z
\]
has at least $3N+1$ elements. Let $h'=h\rs (\cT^{\cG}\cap \bbS_{g(k^*)}^{\cG})$. 
Then for each $p\in \bbN$ there is a large enough $m=m(p)$ such 
that  
either  (i$_X)$ or (ii$_X$) holds, 
and  either (i$_Y$) or (ii$_Y$) holds. 
 We say that such $m$ \emph{decides the $k^*$-type of $p$}. 

Pick $m$ large enough to decide the $k^*$-type of each $p\in [h'(i(l^*)),h'(i(m^*))]\cap Z$. 
Since there are only four different $k^*$-types, by the pigeonhole principle there are $N$
elements of $[h'(i(l^*)),h'(i(m^*))]\cap Z$ with the same $k^*$-type. 
There is therefore  $t^*\in [Z]^{\bar k}$  such that $h_{t^*}$ extends $t'$ and 
all $N$ elements of $t^*\cap [h'(i(l^*)),h'(i(m^*))]$ have the same $k^*$-type. 
This means that $h_n\cup h_{t^*}\subseteq^D C_{g(k^*),i(l^*),i(m^*)}$, 
contradicting  $\psi(t^*)=(k^*,l^*,m^*)$. 

Therefore  there exists $t\in [\bbN]^{\bar k}$ such that for every $k<\bar k$
and all $l<m<\bar k$ such that $[i(l),i(m)]_I\cap \bbS_{g(k)}=\emptyset$ for some 
$n=n(k,l,m)$ we have
\begin{equation*} 
A_{h_n\cup h_t} \subseteq^D C_{g(k),i(l),i(m)}.
\end{equation*}
Then $h_t$ and  $n=\max_{k,l,m} n(k,l,m)$ satisfy \eqref{Eq.1} and this completes the proof. 
\end{proof} 

By Claim~\ref{C.fip} we can find an ultrafilter $\cU$ such that 
 the sets 
 $B_{ij}$ for $i<j$ in $I$ and $C_{gij}$ for $g\in \prod_{n\in\bbN} M_n$ and
$i<j$ in $I$ such that $[i,j]_\phi\cap I=\emptyset$ all belong to $\cU$. 
Let $\ba_i$ be the element of the ultrapower $\prod_{\cU} M_n$
with the representing sequence $\ff_i$ if $i\in I$ and 
with the representing sequence $a_i(n)$, for $n\in\bbN$, if $i\in\bbN$. 
Since the relevant $A_{ki}$ and $B_{ij}$ belong to $\cU$ we have
that $\ba_i$, $i\in I$, is a $\phi$-chain in the ultraproduct. 

In order to check it is weakly $(\aleph_1,\phi)$-skeleton like fix $\bg\in \prod_{\cU} M_n$ and
a representing 
sequence $g\in \prod_n M_n$ of $\bg$. 
 Let $J_{g}= \{\ff_i\colon i\in \bbS\}$. If $i<j$ are  
 such that $[i,j]_I\cap J_g=\emptyset$, 
 then $C_{gij}\in \cU$, which implies 
 that $\prod_{\cU} M_n\models \phi(\ba_i,\bg)\liff \phi(\ba_j,\bg)$ and
 $\prod_{\cU} M_n\models\phi(\bg,\ba_i)\liff\phi(\bg,\ba_j)$, as required. 
 \end{proof} 

\begin{remark} \label{R.I.1} 
As pointed out in Remark~\ref{R.I.0}, the proof of Proposition~\ref{P.construction} can 
be easily modified to obtain $\cU$ such that $\prod_{\cU} M_i$
includes a $\phi$-chain $\cC$ isomorphic to $I$ that satisfies the indiscernibility property (*) 
stronger than being weakly $(\aleph_1,\phi)$-skeleton like 
stated there. In order to achieve this, we only need to add a variant $D_{ijg\psi}$ 
of the set $C_{ijg}$ to the filter basis from Claim~\ref{C.fip}
 for
every $k\in \bbN$, every $k+n$-ary formula $\psi(\vec x,\vec y)$ and every $g\in A^k$. 
Let 
\begin{align*}
D_{ijg\psi}=\{n\colon&
M_n\models \psi(\ff_i(n), g(n))\liff \psi(\ff_j(n), g(n))\}.
\end{align*}
The obvious modification of the proof of Claim~\ref{C.fip} 
shows that the augmented family of sets still has the 
finite intersection property. It is clear that any ultrafilter $\cU$ extending this family is as required.  
\end{remark} 

\section{The proof of Theorem~\ref{T1}}
\label{S.proof.T1}

 Fix a model A of cardinality 
$\leq\fc$ whose theory is  unstable.
By \cite[Theorem~2.13]{Sh:c} the
theory of $A$ has the order property and we can fix $\phi$ in the signature of
$A$ such that $A$ includes arbitrarily long finite $\phi$-chains. Therefore 
Theorem~\ref{T1} is a special case of the following with $A_i=A$ for all $i$. 

\begin{thm}\label{T1+}  Assume CH fails. 
Assume $\phi(\vec x, \vec y)$ is  a  formula and $A_i$, for $i\in \bbN$, 
are models of cardinality $\leq\fc$ such that in $A_i$ there is a $\preceq_\phi$-chain of length $i$. 
Then 
there are $2^{\mathfrak c}$ isomorphism types of models of the form $\pcU A_n$, 
where $\cU$ ranges over nonprincipal ultrafilters on~$\bbN$. 
\end{thm}

\begin{proof}
Since $|A_i|\geq i$ for all $i$, the ultrapower $\prod_{\cU} A$ has cardinality equal to $\fc$ 
whenever $\cU$ is a nonprincipal ultrafilter on $\bbN$. 
By Lemma~\ref{L.many-invariants} , 
 there are $2^{\mathfrak c}$  
  linear orders $I$ of cardinality $ \mathfrak c$ with 
disjoint 
 representing sequences corresponding to (defined) 
 invariants $\inv^{m,\fc}(I)$ (with $m=2$ or $m=3$ 
 depending on wheher $\fc$ is regular or not). 
 Use Proposition~\ref{P.construction} to construct an ultrafilter $\cU(I)$
such that $I$ is isomorphic to a 
weakly $(\aleph_1,\phi)$-skeleton like $\phi$-chain $\cC$ in $\prod_{\cU(I)} A_i$.  
The conclusion follows by Proposition~\ref{P.counting}. 
\end{proof}

\section{Ultrapowers of metric structures}
\label{S.metric} 

\subsection{Metric structures}
In this section we prove a strengthening of Theorem~\ref{T1.m} 
which is the analogue of Theorem~\ref{T1+} for  metric structures.  
First we include the definitions pertinent to understanding the statement 
of Theorem~\ref{T1.m}. 
Assume $(A,d,f_0,f_1,\dots,R_0,R_1,\dots)$ is a metric structure. 
Hence $d$ is a complete metric on $A$ such that the diameter of $A$ is equal to 1,  
each $f_i$ is a function from some finite power of $A$ into $A$, 
and each $R_i$ is a function from 
a finite power of $A$ into $[0,1]$. All $f_i$ and all $R_i$ are
 required to be uniformly continuous with respect to $d$, with  a 
 fixed modulus of uniform continuity  (see \cite{BYBHU} or \cite[\S 2]{FaHaSh:Model2}).

If~$\cU$ is an ultrafilter on $\bbN$ then on $A^{\bbN}$ we define 
a quasimetric $d_{\cU}$ by letting, for $\ba=(a_i)_{i\in \bbN}$ 
and $\bb=(b_i)_{i\in \bbN}$, 
\[
d_{\cU}(\ba,\bb)=\{i\in \bbN\colon \lim_{i\to \cU}d(a_i,b_i)\}.
\]
Identify pairs $\ba$ and $\bb$ such that $d_{\cU}(\ba,\bb)=0$. 
The uniform continuity implies that $f_n(\ba)=\lim_{i\to \cU} f_n(a_i)$ 
and $R_n(\ba)=\lim_{i\to \cU}R_n(b_i)$  
are uniformly continuous functions with respect to the quotient metric. 
The  quotient structure is denoted by $\prod_{\cU} (A,d,\dots)$ (or shortly $\prod_{\cU} A$ if 
the signature is clear from the context)
and called the \emph{ultrapower of $A$ associated with $\cU$}. 
An ultraproduct of metric structures of the same signature 
 is defined analogously. 

The assumption that the metric $d$ is finite is clearly necessary 
in order to have $d_{\cU}$ be a metric. However, one can 
show that the standard ultrapower constructions of C*-algebras and 
of II$_1$ factors can  essentially be considered as special 
cases of the above definition (see 
\cite{FaHaSh:Model2} for details). These two constructions served as a motivation
for our work (see \S\ref{S.Applications}). 

More information on  the logic of metric structures
is given in~\cite{BYBHU}, and \cite{FaHaSh:Model2} contains an exposition of 
its variant suitable for   C*-algebras and II$_1$ factors. 

Let $A=(A,d,\dots) $ be a metric structure. 
 Interpretations of formulas are functions uniformly continuous with respect to $d$, 
and the value of an $n$-ary  formula $\psi$ at an $n$-tuple  $\vec a$ is denoted by 
\[
\psi(\vec a)^A.
\] 
We assume that the theory of $A$ is unstable, 
and therefore by \cite[Theorem~5.4]{FaHaSh:Model2}
it has the order property. Fix $n$ and a $2n$-ary formula $\phi$ that witnesses the
order property of the theory 
of $A$. 
Define the relation $\preceq_\phi$ on every model such that $\phi$ is a formula in 
 its signature  
by letting $\vec a\preceq_\phi \vec b$ if and only if 
\[
\phi(\vec a,\vec b)=0\text{ and } \phi(\vec b,\vec a)=1. 
\]
Theorem~\ref{T1.m} is a consequence of the following.

\begin{thm}\label{T1.m+}  Assume CH fails. 
Assume $\phi(\vec x, \vec y)$ is  a  formula and $A_i$, for $i\in \bbN$, 
are metric structures of 
 cardinality $\leq\fc$ of the same signature 
 such that in $A_i$ there is a $\preceq_\phi$-chain of length $i$. 
Then 
there are $2^{\mathfrak c}$ isometry types of models of the form $\pcU A_n$, 
where $\cU$ ranges over nonprincipal ultrafilters on~$\bbN$. 
\end{thm}

The proof proceeds along the same lines as the proof of Theorem~\ref{T1+} and we shall only
outline the novel elements, section by section.

\subsection{Combinatorics of the invariants} 
 For $\vec a\in A^n$ and  $\bar b\in A^n$
 write 
\[
\tp_\phi(\vec a/\vec b)=\langle \phi(\vec a,\vec b)^A,\phi(\vec b,\vec a)^A\rangle.
\] 
  For  $\vec a\in A^n$ and $X\subseteq A^n$, let $\tp_\phi(\vec a/X)$  be the function 
from $ X$ into $[0,1]^2$ defined by 
\[
\tp_\phi(\vec a/X)(\vec b)=
\tp_\phi(\vec a/\vec b).
\]
A \emph{$\phi$-chain} $\cC$ in $A$ is a subset of $A^n$ linearly ordered by $\preceq_\phi$. 
 The notation and terminology such as $[\vec a,\vec b]_\phi$ have exactly the same
interpretation as in~\S\ref{S.OP}. 

\begin{definition} \label{Def.wsl.m}
A $\phi$-chain $\cC$ is \emph{weakly $(\aleph_1,\phi)$-skeleton like in $A$} if 
for every $\vec a\in A^n$ there is a countable $\cC_{\vec a}\subseteq \cC$
such that for all $\vec b$ and $\vec c$ in $\cC$ satisfying
\[
(-\infty,\vec b]_\phi\cap \cC_{\vec a}=
(-\infty,\vec c]_\phi\cap \cC_{\vec a}
\]
we have $\tp_\phi(\vec a/\vec b)=\tp_\phi(\vec a/\vec c)$. 
\end{definition}

Note that $(\cC,\preceq_\phi)$ is an honest (discrete) linear ordering. Because of this
a number of the proofs in the discrete case work in the metric case unchanged. 
In particular, 
 Lemma~\ref{L.wf}, Lemma~\ref{L0},  Lemma~\ref{L1} and Lemma~\ref{L.2.6}
 are true with the new definitions and the old proofs. 
Definition~\ref{Def.Defines} and the definition of $\cC[B,\vec c]$ are
 transferred to the metric case unmodified, 
 using the new definition of $\tp_\phi$. 
As a matter of fact, the analogue of  Remark~\ref{R.I.-1} applies in the metric context. 
That is, even if weakly $(\aleph_1,\phi)$-skeleton like is defined by requiring only that
(with $\bar a$, $\cC_{\bar a}$, $\bar b$ and $\bar c$ as in 
Definition~\ref{Def.wsl.m}) we only have
\[
\bar a\leq_\phi \bar b\text{ if and only if } 
\bar a\leq_\phi \bar c
\]
and 
\[
\bar b\leq_\phi \bar a\text{ if and only if } 
\bar c\leq_\phi \bar a
\]
then all of the above listed lemmas remain true, with the same proofs, in the metric context. 
However, Lemma~\ref{L.disj.1.m+} 
below requires the original, more restrictive, notion of weakly ($\aleph_1,\phi$)-skeleton like.

\subsection{Defining an  invariant over a  submodel} 
Definition~\ref{Def.Defines} is unchanged. 
The statement and the proof of  
  Lemma~\ref{L.tp-phi} remain unchanged. 
However, in order to invoke it in the proof of the metric analogue of Lemma~\ref{L.tp} 
we shall need Lemma~\ref{L.3.9.m} below. For a metric structure $B$ 
its \emph{character density}, the smallest cardinality of a dense subset, 
is denoted by~$\chi(B)$.  Note that $\chi(A)\geq |\cC|$ for every $\phi$-chain $\cC$ in
$A$, since each $\phi$-chain is necessarily discrete. 

\begin{lemma} \label{L.3.9.m} 
Assume $\cC=\langle a_i\colon i\in I\rangle$ is a $\phi$-chain 
that is weakly $(\aleph_1,\phi)$-skeleton like in a metric structure $A$. 
Assume $B$ is an elementary submodel of  $A$ and $\vec a\in \cC\setminus B^n$ is  such that 
\[
\cf(\cC[B,\bar a],\leq_{\lnot\phi})>\chi(B). 
\]
Then there is $\bar c\in \cC[B,\bar a]$ such that for 
all $\bar d\in \cC[B,\bar a]\cap (-\infty,\bar c]_\phi$
we have $\tp_\phi(\bar d/B)=\tp_\phi(\bar c/B)$. 
\end{lemma} 

\begin{proof} Pick a dense $B_0\subseteq B$ of cardinality $\chi(B)$. 
Let $\bar c\in \cC[B,\bar a]$ be such that 
\[
\cC[B,\bar a]\cap \bigcup\{\cC_{\bar b}\colon \bar b\in B_0^n\}\cap 
 (-\infty,\bar c]_\phi=\emptyset. 
 \]
 Then for every $\bar d\in \cC[B,\bar c]\cap (-\infty,\bar c]_\phi$ and every $\bar b\in B_0^n$ we
 have that $[\bar d,\bar c]_\phi\cap \cC_{\bar b}=\emptyset$, and therefore 
 $\tp_\phi(\bar c/\bar b)=\tp_\phi(\bar d/\bar b)$. 
Since the maps $\bar x\mapsto \tp_\phi(\bar c/\bar x)$ and 
$\bar x\mapsto \tp_\phi(\bar d/\bar x)$ are continuous, they agree on all of $B^n$
and therefore $\tp_\phi(\bar c/B)=\tp_\phi(\bar d/B)$. 
\end{proof}

\subsection{Representing invariants} 
 The definition of $\INV^{m,\lambda}(A,\phi)$
  from \S\ref{S.RI} transfers to the metric  context verbatim, and
  Lemma~\ref{L.J} and its proof are unchanged. 

\subsection{Counting the number of invariants over a model} 
 Lemma~\ref{L.trivial} is  unchanged but  
 Lemma~\ref{L.disj.1} needs to be modified, since the right 
 analogue of cardinality of a model is its character density. 
 
 \begin{lemma} \label{L.disj.1.m} 
 For $A,\phi,m$ as usual 
every  set of disjoint representing sequences of invariants in 
$\INV^{m,\chi(A)}(A,\phi)$ has size at most $\chi(A)$. 
\end{lemma} 

\begin{proof} In this paper we shall only need the trivial case when $\chi(A)=|A|=\fc$, 
but the general case is needed in \cite{FaKa:NonseparableII}. It will follow immediately 
from the proof of Lemma~\ref{L.disj.1} with   Lemma~\ref{L.disj.1.m+} below applied in the
right moment. 
\end{proof} 

 \begin{lemma} \label{L.disj.1.m+} 
 For $A,\phi,m$ as usual
 and an elementary submodel $B$ of $A$ 
 there are at most $\chi(A)$ distinct $(A,B,\phi,m)$-invariants. 
\end{lemma} 

\begin{proof} Let $\lambda=\chi(A)$. 
Let $h\colon \bbR\to [0,1]$ be the continuous function such that  
 $h(x)=0$ for $x\leq 1/3$, $h(x)=1$ for $x\geq 2/3$, and $h$ linear on $[1/3,2/3]$. 
 Let $\psi=h\circ \phi$. 
 
 Note that every $\phi$-chain is a $\psi$-chain. 
 Also, $\phi(\bar a_1,\bar b_1)=\phi(\bar a_2,\bar b_2)$ implies
 $\psi(\bar a_1,\bar b_1)=\psi(\bar a_2,\bar b_2)$, and therefore 
 every 
 weakly $(\aleph_1,\phi)$-skeleton like $\phi$-chain is 
 weakly $(\aleph_1,\psi)$-skeleton like, with the same witnessing sets $\cC_{\bar a}$. 
  This implies the following, for every elementary submodel $B$ of $A$ and $m\in \bbN$. 
  \begin{enumerate}
  \item [(*)] If $\bar c\in A^n$ defines the 
  $(A,B,\phi,m)$-invariant $\bd$ 
  then $\bar c$ defines the $(A,B,\psi,m)$-invariant $\bd$. 
  \end{enumerate}
  Denote the sup metric on $A^n$ by $d^n$. 
Since $\phi^A$ is a uniformly continuous function, there is $\delta>0$ sufficiently small so 
that $d^n(\bar c_1,\bar c_2)<\delta$ implies $|\phi(\bar a,\bar c_1)-\phi(\bar a,\bar c_2)|<1/3$
for all $\bar a$. Therefore we have the following. 
\begin{enumerate}
\item [(**)] For every $\bar a\in A^n$ we have
that $\bar a\leq_\phi \bar c_1$ implies $\bar a\leq_\psi \bar c_2$, 
and $\bar a\leq_\phi \bar c_2$ implies $\bar a\leq_\psi \bar c_1$. 
\end{enumerate}
Assume $B$ is an elementary submodel of $A$ and 
 $\bar c_i$ defines the $(A,B,\phi,m)$-invariant $\bd_i$, for $i=1,2$. 
By (*) we have that $\bar c_i$ defines the $(A,B,\psi,m)$-invariant $\bd_i$, for $i=1,2$. 
If $d^n(\bar c_1,\bar c_2)<\delta$ then (**) implies  $\bd_1=\bd_2$. 
\end{proof} 

Proposition~\ref{P.counting} applies in the metric case literally. 

\subsection{Construction of ultrafilters} 
It is the construction of the ultrafilter in \S\ref{S.Construction} that requires the most drastic 
modification.  Although  the statement of Proposition~\ref{P.construction} transfers unchanged, 
the proof of its analogue, Proposition~\ref{P.construction.m},  requires new ideas.

 \begin{prop}\label{P.construction.m} 
Assume $\phi(\vec x, \vec y)$ is  a  formula and $M_i$, for $i\in \bbN$, 
are metric structures 
of the same signature such that in $M_i$ there is a $\preceq_\phi$-chain of length $i$. 
Assume  $I$ is a linear order of cardinality $\leq \fc$. Then there is an 
ultrafilter $\cU$ on $\bbN$ such that  $\prod_{\cU} M_n$ 
includes a $\phi$-chain $\{\ba_i\colon i\in I\}$ 
that is  weakly $(\aleph_1,\phi)$-skeleton like.  
\end{prop} 

\begin{proof}  Like in the proof of Proposition~\ref{P.construction}, we assume $\phi$ is a 
binary formula in order to simplify the notation. Fix  a $\phi$-chain 
$a_i(n)$, for $0\leq i<n$, in~$M_n$. Like in \S\ref{S.Construction} fix an independent family 
$\cG$ of size $\fc$ and a filter $D$ such that $\cG$ is independent with respect to $D$ and
$D$ is a maximal filter with this property. Define  
$\ccG$, $\FI(\ccG)$ and $\FI_s(\ccG)$ exactly as in \S\ref{S.Construction}. 
Since the diameter of each  $M_n$ is $\leq 1$, each element of $\prod_n M_n$ is 
a representing sequence of an element of the ultrapower. 
Claim~\ref{C.const.1} is modified as follows. 

\begin{claim}\label{C.const.1.m} 
For every $g\in \prod_{n\in \bbN} M_n$ there is a countable set $\bbS_g\subseteq I$
such that  for all $l\in\bbN$ and all $r\in \bbQ\cap [0,1]$ all  sets of the form
\begin{align*} 
X_{g,l,r}&=\{n\colon  \phi(a_l(n),g(n))^{M_n}<r\}\\
Y_{g,l,r}&=\{n\colon  \phi(g(n),a_l(n))^{M_n}<r\}
\end{align*} 
are supported by $\bbS_g$.  
\end{claim} 
\begin{proof} 
Since there are only countably many relevant sets, this is an 
immediate consequence of Lemma~\ref{L.Cohen}. 
\end{proof} 
For  $i<j$ in $ I$  the definitions of  sets 
\[
B_{ij}=\{n\colon \ff_i(n)\preceq_\phi \ff_j (n)\} 
\]
is unchanged, but we need to modify the definition of $C_{gij}$. 
For $g\in \prod_{n\in\bbN} M_n$, $i<j$ in $I$ such that 
$[i,j]_i\cap \bbS_g=\emptyset$
and $\e>0$ let
\begin{align*}
C_{gij\e}=\{n\colon&
| \phi(\ff_i(n), g(n))^{M_n}- \phi(\ff_j(n), g(n))^{M_n}|<\e\\
&\text{ and } 
|\phi(g(n),\ff_i(n))^{M_n}- \phi(g(n),\ff_j(n))^{M_n}|<\e\}.
\end{align*}

\begin{claim} \label{C.fip.m} 
The family of all sets 
$B_{ij}$ for $i<j$ in $I$ and $C_{gij\e}$ for $g\in \prod_{n\in\bbN} M_n$, 
$i<j$ in $I$ such that $[i,j]_i\cap \bbS_g=\emptyset$ and $\e>0$  has the finite intersection 
property. 
\end{claim} 

\begin{proof} It will suffice to show that for $\bar k\in\bbN$, $i(0)<\dots <i(\bar k-1)$ in $I$, and 
$g(0),\dots, g(\bar k-1)$ in $\prod_{n\in\bbN} M_n$ and $\e>0$ the set 
\begin{multline*}
\bigcap_{l<m<\bar k} B_{i(l),i(m)}\cap\\
\bigcap\{C_{g(k),i(l), i(m),\e}\colon k<\bar k, l<m<\bar k, \text{ and } 
[i(l),i(m)]_I\cap \bbS_{g(k)}=\emptyset\}
\end{multline*}
is nonempty. Pick $M\in\bbN$ such that $M>2/\e$. 
Let
\[
\bbS=\bigcup_{k<\bar k}\bbS_{g(k)}.
\]
Write $\cT=\{{i(k)}\colon k<\bar k\}$, also 
$\cT^{\cG}=\{f_i\colon i\in \cT\}$ and $\bbS^{\cG}=\{f_i\colon i\in \bbS\}$. 

 Pick $h_m$, for $m\in\bbN$, in $\FI(\bbS^{\cG}\setminus \cT^{\cG})$ so that
\begin{enumerate}
\item $h_m\subseteq h_{m+1}$ for all $m$ and 
\item For all $h\in \FI(\cT^{\cG})$, all $l\in\bbN$, 
 and all $k<\bar k$
 there exist $r$ and $s$ in $\bbN$ such that $0\leq r\leq M$, $0\leq s\leq M$ and
 for all but finitely many $m$ we have 
\begin{enumerate}
\item [(i$_X$)] $(\forall^D n\in A_{h_m\cup h})|\phi(a_l(n),g(k)(n))^{M_n}-r/M|<\e/2$ and
\item [(i$_Y$)] $(\forall^D n\in  A_{h_m\cup h})| \phi(g(k)(n),a_l(n))^{M_n}-s/M|<\e/2$. 
\end{enumerate}
\end{enumerate}
The construction of $h_m$ is essentially the same as in 
the proof of Claim~\ref{C.fip}, except that it uses Claim~\ref{C.const.1.m} 
in place of Claim~\ref{C.const.1}. 

In order to complete the proof we need to show 
that there exist $h\in \FI(\cT^{\cG})$ and $n$ such that 
\begin{multline}\label{Eq.1.m}  
A_{h_n\cup h}\subseteq ^D 
\bigcap_{l<m<\bar k} B_{i(l),i(m)}\cap\\
\bigcap\{C_{g(k),i(l), i(m),\e}\colon k<\bar k, l<m<\bar k, \text{ and } 
[i(l),i(m)]_I\cap \bbS_{g(k)}=\emptyset\}.
\end{multline}
In order to have $A_{h_n\cup h}\subseteq^D B_{i(l),i(m)}$ it is necessary and sufficient 
to have  $h(i(l))<h(i(m))$. We shall therefore consider only $h$ that are \emph{increasing} 
in this sense. An increasing function in $\FI(\cT^{\cG})$ is uniquely determined by its range. 
For $t\in [\bbN]^{\bar k}$ let $h_t$ denote the increasing function in $\FI(\cT^{\cG})$ 
whose range is equal to $t$.

Assume for a moment that for every $t\in [\bbN]^{\bar k}$ there are
$k,l,m$ such that   for all $n$ we have 
$A_{h_n\cup h_t} \not\subseteq^D  C_{g(k),i(l),i(m),\e}$ and therefore by the choice 
of the sequence  $\{h_n\}$ that 
\begin{equation*}  
A_{h_n\cup h_t} \cap C_{g(k),i(l),i(m),\e}=^D \emptyset. 
\end{equation*}
For $t\in [\bbN]^{\bar k}$ let $\psi(t)$ be the lexicographically minimal triple $(k,l,m)$ 
such that this holds for a large enough $n$. By Ramsey's theorem, there are an infinite 
$Z\subseteq \bbN$ and a triple $(k^*,l^*,m^*)$ such that for every $t\in [\bbN]^{\bar k}$ 
we have $A_{h_n\cup h_t} \cap C_{g(k),i(l),i(m),\e}=^D \emptyset$.

Let $N=|[i(l^*),i(m^*)]_I\cap \cT|$ and 
find $t\in [Z]^{\bar k}$ such that the set
\[
[h_t(i(l^*)),h_t(i(m^*))]\cap Z
\]
has at least $(M^2+2M)N+1$ elements. Let $h'=h\rs (\cT^{\cG}\cap \bbS_{g(k^*)}^{\cG})$. 
Then for each $p\in \bbN$ there are  a large enough $m=m(p)$ such that for some 
$r=r(p)$ and $s=s(p)$ we have
\[
(\forall^D n\in A_{h_m\cup h})|\phi(a_l(n),g(k)(n))^{M_n}-r/M|<\e/2
\]
 and
\[
(\forall^D n\in  A_{h_m\cup h})| \phi(g(k)(n),a_l(n))^{M_n}-s/M|<\e/2.
\] 
     We say that such $m$ \emph{decides the $k^*$-type of $p$}. 
Pick $m$ large enough to decide the $k^*$-type of each $p\in [h'(i(l^*)),h'(i(m^*))]\cap Z$. 
Since there are only $(M+1)^2$ different $k^*$-types, by the pigeonhole principle there are $N$
elements of $[h'(i(l^*)),h'(i(m^*))]\cap Z$ with the same $k^*$-type. 
There is therefore  $t^*\in [Z]^{\bar k}$  such that $h_{t^*}$ extends $t'$ and 
all $N$ elements of $t^*\cap [h'(i(l^*)),h'(i(m^*))]$ have the same $k^*$-type. 
This means that $h_n\cup h_{t^*}\subseteq^D C_{g(k^*),i(l^*),i(m^*),\e}$, 
contradicting  $\psi(t^*)=(k^*,l^*,m^*)$. 

Therefore  there exists $t\in [\bbN]^{\bar k}$ such that for every $k<\bar k$
and all $l<m<\bar k$ such that $[i(l),i(m)]_I\cap \bbS_{g(k)}=\emptyset$ for some 
$n=n(k,l,m)$ we have
\begin{equation*} 
A_{h_n\cup h_t} \subseteq^D C_{g(k),i(l),i(m),\e}.
\end{equation*}
Then $h_t$ and  $n=\max_{k,l,m} n(k,l,m)$ satisfy \eqref{Eq.1.m}. 
\end{proof}

Let $\cU$ be any ultrafilter that extends the family of sets 
from the statement of Claim~\ref{C.fip.m}. 
Since $M_n$ are assumed to be bounded metric spaces, each 
$f_i$ is a representing sequence of an element 
 of the ultraproduct
$\prod_{\cU} M_n$. Denote this element by $\ba_i$ and let 
$\cC$ denote $\langle \ba_i\colon i\in I\rangle$. 
Since $B_{i,j}\in \cU$ for all $i<j$ in $I$, $\cC$ is a $\phi$-chain isomorphic to $I$. 
For $\bb\in \prod_{\cU} M_n$ fix its representing sequence $g$ and 
let $\cC_{\bb}\subseteq \cC$ be 
$\{\ba_i\colon i\in \bbS_{g}\}$. Since 
$C_{g,i,j,\e}\in \cU$ whenever $[i,j]\cap \bbS_g=\emptyset$  and $\e>0$, 
we conclude that $\cC$
 is a weakly $(\aleph_1,\phi)$-skeleton like $\phi$-chain 
as in the  proof in \S\ref{S.Construction}. 
\end{proof}

\subsection{The proof of Theorem~\ref{T1.m+}}
Compiling  the above facts into the proof of Theorem~\ref{T1.m+} 
proceeds  exactly like in  \S\ref{S.proof.T1}. 

\begin{remark} \label{R.I.2} Remark~\ref{R.I.1} applies to 
Proposition~\ref{P.construction.m} in place of Proposition~\ref{P.construction} verbatim. 
\end{remark} 

\section{Types with the order property}
\label{S.Local} 
In this section we prove local versions of Theorem~\ref{T1} and Theorem~\ref{T1.m}
in which the $\phi$-chain is contained in the set of $n$-tuples  realizing a prescribed
type $\bt$ (the definition of a type in the logic of metric 
structures is given below). We will make use of this in case when $\bt$ is the set of 
all $n$-tuples all of whose entries realize a given  $1$-type,  and 
the set of these realizations  is a substructure. In order to conclude that 
a $\phi$-chain is still a $\phi$-chain when evaluated in this substructure,
we will consider a formula $\phi$ that is quantifier-free. 
Throughout this section we assume $A$ is  
a model, 
 $\phi(\vec x, \vec y)$ is  a $2n$-ary   formula in the same 
 signature and~$\bt$ is an $n$-ary type over $A$.

Although the motivation for this section comes from the metric case, we shall 
first provide the definitions and results in the classical case of discrete models. 
An $n$-ary  type $\bt$ over $A$ has the \emph{order property} if there exists 
a $2n$-ary formula $\phi$ such that for 
every finite $\bt_0\subseteq \bt$ and for every $m\in \bbN$ there exists
a $\phi$-chain of length $m$ in $A$ all of whose elements realize~$\bt_0$.

\begin{prop}\label{P.local.1} 
Assume $A$ is countable and type $\bt$ over 
$A$ has the order property, as witnessed by $\phi$. 
 Assume  $I$ is a linear order of cardinality $\leq \fc$. Then there is an 
ultrafilter $\cU$ on $\bbN$ such that  $\prod_{\cU} A$ 
includes a weakly $(\aleph_1,\phi)$-skeleton like 
$\phi$-chain isomorphic to $I$ 
consisting of $n$-tuples  realizing $\bt$. 
\end{prop}

\begin{proof} Since $\bt$ is countable we may write it as a union 
of finite subtypes, $\bt=\bigcup_{i\in \bbN} \bt_i$. Let $a_i(k)$, for $0\leq i<k$, 
be a $\preceq_{\phi}$ chain in $A$ of elements realizing $\bt_k$. 
Let $\cG$ be an independent family of functions of cardinality $\fc$. 
Unlike the proof of  Proposition~\ref{P.construction}, we cannot identify 
$\cG$ with functions in $\prod_k \{a_i(k)\colon i<k\}$, since we cannot
assume $a_i(k)=a_i(l)$ for all $i<\min(k,l)$. Therefore 
to each $g\in \cG$ we associate a function $\hat g$ such that
\[
\hat g(k)=a_{g(k)}(k) 
\]
if $g(k)<k$ and $\hat g(k)=a_{k-1}(k)$, otherwise. Then 
by the Fundamental Theorem of Ultraproducts $\hat g$ is a representing sequence of an 
element that realizes $\bt$. 
The rest of the  proof is identical to the proof of Proposition~\ref{P.construction}. 
\end{proof} 

In order to state the metric version of Proposition~\ref{P.local.1} we 
import some notation from 
 \cite{Fa:Relative} and \cite{FaHaSh:Model1}. Given  $0\leq \e<1/2$   
 define relation $\preceq_{\phi,\e}$ on $A^n$ via
\[
\vec a_1\preceq_{\phi,\e} \vec a_2 \mbox{ if } \phi(\vec a_1,\vec a_2)\leq \e 
\text{ and } \phi(\vec
a_2,\vec a_1)\geq 1-\e
\]
Note that  $\preceq_{\phi,0}$ coincides with 
 $\preceq_\phi$. A \emph{$\phi,\e$-chain} 
is defined in a natural way.

We shall now define a type in the logic of metric structures, following \cite{BYBHU} and
\cite[\S 4.3]{FaHaSh:Model2}. 
 A \emph{condition} over a model $A$ is
an expression of the form $\phi(\vec x,\vec a) \leq r$ where
$\phi$ is a formula, $\vec a$ is a tuple of elements of $A$ 
and $r \in \bbR$.  
A type $\bt$ over $A$ is a set of conditions over $A$. 
A condition $\phi(\vec x,\vec a)\leq r$ is \emph{$\e$-satisfied} in $A$ by $\vec b$ if
$\phi(\vec b,\vec a)^A\leq r+\e$. 
Clearly a condition is satisfied by $\vec b$ in $A$
if and only if it is $\e$-satisfied by $\vec b$ for all $\e>0$. 
A type $\bt$ is \emph{$\e$-satisfied} by $\bar b$ if all  conditions
in $\bt$ are $\e$-satisfied by $\bar b$.  

An $n$-ary  type $\bt$ over a metric structure
 $A$ has the \emph{order property} if there exists 
a $2n$-ary formula $\phi$ such that for 
every finite $\bt_0\subseteq \bt$ and for every $m\in \bbN$ there exists
a $\phi,1/m$-chain of length $m$ in $A$
consisting of $n$-tuples each of which $1/m$-satisfies $\bt_0$. 

\begin{prop}\label{P.local.1.m} 
Assume $A$ is separable metric structure and type $\bt$ 
over~$A$ has the order property, as witnessed by $\phi$. 
 Assume  $I$ is a linear order of cardinality $\leq \fc$. Then there is an 
ultrafilter $\cU$ on $\bbN$ such that  $\prod_{\cU} A$ 
includes a weakly $(\aleph_1,\phi)$-skeleton like 
$\phi$-chain isomorphic to $I$ 
 and 
consisting of $n$-tuples  realizing $\bt$. 
\end{prop} 

\begin{proof}
For elements $\ba$ and $\bb$ of $\prod_{\cU} A$ and their 
representing sequences $(a_i)_{i\in \bbN}$ 
and $(b_i)_{i\in \bbN}$  we have $\ba\preceq_\phi \bb$ in $\prod_{\cU} A$ 
if and only if $\{i\colon a_i\preceq_{\phi,\e} b_i\}\in \cU$ for every $\e>0$. 
Modulo this observation, the proof is identical to the proof of Proposition~\ref{P.local.1}. 
\end{proof} 

In order to prove versions of Proposition~\ref{P.local.1} and Proposition~\ref{P.local.1.m}
for uncountable (respectively, nonseparable) structures 
we shall need the following well-known lemma. 

\begin{lemma} \label{L.meager} 
Assume $D$ is a meager filter on $\bbN$ extending the Frech\'et filter. 
Then there is a family~$\cG_D$ of cardinality $\fc$ 
of functions in $\bbN^{\bbN}$ that is independent mod $D$. 
\end{lemma}

\begin{proof} Let $\cG$ be a family of cardinality $\fc$ that is independent mod the Fr\'echet filter
(\cite[Appendix, Theorem~1.5(1)]{Sh:c}). 
Since $D$ is meager there is a surjection $h\colon \bbN\to \bbN$
such that the $h$-preimage of every finite set is finite and the 
$h$-preimage of every infinite set is $D$-positive  (see e.g., \cite{BarJu:Book}). 
Then  $\cG_D=\{h\circ f\colon f\in \cG\}$ is independent mod $D$ because
the $h$-preimage of every infinite set is $D$-positive. 
\end{proof} 

Again $A,\phi$ and $\bt$ are as above and $A^{<\bbN}$ denotes the 
set of all finite sequences of elements of $A$. Note that $A$ is 
not assumed to be countable. 

\begin{prop}\label{P.local.2} 
Let  $A$ be a model and let $\bt$ is 
a type over $A$. 
Assume there is a function $\bh\in \prod_{k\in \bbN}  A^{k\cdot n}$ 
such that the sets
\[
X[\bt_0,k]=\{i\colon \bh(i)\text{ is a $\phi$-chain  
of $n$-tuples satisfying
 $\bt_0$}\}
\]
for $\bt_0\subseteq \bt$ finite and $k\in \bbN$ generate a meager filter extending the 
Frech\'et filter. 

 Assume  $I$ is a linear order of cardinality $\leq \fc$. Then there is an 
ultrafilter $\cU$ on $\bbN$ such that  $\prod_{\cU} A$ 
includes a $\phi$-chain $\{\ba_i\colon i\in I\}$ 
that is  weakly $(\aleph_1,\phi)$-skeleton like and 
consists of elements realizing $\bt$. 
\end{prop} 

\begin{proof} Let $D_0$ denote the filter generated by all $X[\bt_0,k]$ 
for $\bt_0\subseteq \bt$ finite and $k\in \bbN$. By Lemma~\ref{L.meager} 
there is a family $\cG_0$ of cardinality $\fc$ that is independent mod $D$. 
For each $k\in \bbN$ enumerate the $\phi$-chain $\bh(k)$ as
$a_i(k)$, $i<k$. 
Like in the proof of Proposition~\ref{P.local.1} for $g\in \cG_0$ define
$\hat g\in A^{\bbN}$ by 
$\hat g(k)=a_{g(k)}(k)$ if $g(k)<k$ and $a_{k-1}(k)$ otherwise. 

The construction described in the proof of Proposition~\ref{P.construction} 
results in $\cU$ such that all elements of the resulting $\phi$-chain $\ba_i$, for 
$i\in I$, realize $\bt$. 
\end{proof} 

The proof of the following metric version is identical 
to the proof of Proposition~\ref{P.local.2}. Note that $A$ is not assumed to be separable. 

\begin{prop}\label{P.local.2.m} 
Let  $A$ be a metric structure and let $\bt$ is 
a type over $A$. 
Assume there is a function $\bh\in \prod_{k\in \bbN}  A^{k\cdot n}$ 
such that the sets
\begin{align*}
X[\bt_0,k]=\{i\colon& \bh(i)\text{ is a $\phi,1/k$-chain  
consisting of $n$-tuples  $1/k$-satisfying $\bt_0$}\}
\end{align*}
for $\bt_0\subseteq \bt$ finite and $k\in \bbN$ generate a meager filter
 extending the 
Frech\'et filter. 

 Assume  $I$ is a linear order of cardinality $\leq \fc$. Then there is an 
ultrafilter $\cU$ on $\bbN$ such that  $\prod_{\cU} A$ 
includes a $\phi$-chain $\{\ba_i\colon i\in I\}$ 
that is  weakly $(\aleph_1,\phi)$-skeleton like and 
consists of elements realizing $\bt$. \qed
\end{prop}

\section{Applications} 
\label{S.Applications}

Recall that $\Alt(n)$ is the alternating group on $\{0,\dots, n-1\}$. 
The following is the main result of \cite{ElHaScTh} (see also \cite{Th:On}).  

\begin{thm}[Ellis--Hachtman--Schneider--Thomas] 
If CH fails then there are $2^{\fc}$ ultrafilters on $\bbN$ such that the 
ultraproducts $\prod_{\cU} \Alt(n)$ are pairwise nonisomorphic. 
\end{thm} 
\begin{proof} Let $\phi(x_1,x_2,y_1,y_2)$ be the formula asserting that   
$x_1y_2=y_2x_1$ and $x_2y_1\neq y_1x_2$. 
It is then easy to see that for all natural numbers $k\geq 2n+4$
the group $\Alt(k)$ includes  a $\phi$-chain of length~$n$. 
Therefore the conclusion follows by Theorem~\ref{T1+}. 
\end{proof} 


%
%

\subsection{Applications to operator algebras} \label{S.appl.OA} 
 Theorem~\ref{T1.m} and Theorem~\ref{T1.m+} 
 were stated  and proved for the  case of  bounded metric structures. 
However, the original  motivation for the present paper came from a question 
about the of ultrapowers of C*-algebras and II$_1$ factors stated 
in early versions of \cite{FaHaSh:Model1} and \cite{FaHaSh:Model2}. 
An excellent reference for operator algebras is \cite{Black:Operator}. 

In the following  propositions and accompanying discussion we deal
 with the ultrapower constructions  for 
 C*-algebras and II$_1$-factors, as well as 
the associated relative commutants. Although Theorem~\ref{T1.m} was proved 
for bounded metric structures, it applies to the context of C*-algebras
and II$_1$ factors. Essentially, one applies the result to the unit ball of the given 
algebra.  
All the pertinent definitions can be found in \cite{FaHaSh:Model1} or \cite{FaHaSh:Model2}.

The classes of C*-algebras
and of II$_1$ factors are axiomatizable in the logic of metric structures. 
Both proofs can be found in 
\cite[\S3]{FaHaSh:Model2}, and 
the (much more difficult) II$_1$ factor case was first proved in \cite{BYHJR}, using a rather different 
axiomatization from the one given in \cite{FaHaSh:Model2}.   
Extending results of \cite{GeHa} and \cite{Fa:Relative}, 
in \cite[Lemma~5.2]{FaHaSh:Model1}
 it was also 
proved that the class of infinite dimensional C*-algebras has the order property, as 
witnessed by the formula
\[
\phi(x,y)=\|xy-x\|.
\]
Assume $a_i$, $i\in \bbN$, is a sequence of positive operators of norm one
such that $a_i-a_j$ is positive and of norm one whenever $j<i$. Then 
this sequence forms a $\preceq_\phi$-chain. Such a sequence exists in every 
infinite-dimensional C*-algebra (see the proof of 
\cite[Lemma~5.2]{FaHaSh:Model1}). Ä
Note that it is important to have this $\preceq_\phi$-chain 
inside the unit ball of the algebra. 
In \cite[Lemma~5.2]{FaHaSh:Model1} it was also proved that the relative commutant type 
(see below for the definition) of 
every infinite-dimensional C*-algebra has the order property, and that this is witnessed 
by the same $\phi$ as above.

In \cite[Lemma~3.2 (3)]{FaHaSh:Model1} it was proved that the class of II$_1$ factors
has the order property, as witnessed by the formula
\[
\psi(x_1,y_1,x_2,y_2)=\|x_1 y_2- y_2 x_1 \|_2. 
\]
It was also proved in \cite[Lemma~3.4]{FaHaSh:Model1} that the relative commutant type
(see below) of any II$_1$ factor has the order property, as witnessed by $\psi$ above. 
We emphasize that, similarly to  the case of C*-algebras, an arbitrarily long finite  $\psi$-chain 
can  be found inside the unit ball of the algebra. 
This is necessary in order to have the proof work. 
Note that without this requirement even~$\bbC$ includes an infinite $\psi$-chain, 
although $\bbC$ clearly does not have the order property. 

Recall that two  C*-algebras are (algebraically) isomorphic if and only if they 
are isometric, and that the same  applies to  II$_1$ factors. 
The following is a quantitative improvement to the  results of \cite{GeHa}, \cite{Fa:Relative} 
  (for C*-algebras) and~\cite{FaHaSh:Model1} (for II$_1$ factors).

\begin{prop} \label{C.x} 
Assume $A$ is a separable infinite-dimensional C*-algebra or a separably 
acting II$_1$-factor. 
If the Continuum Hypothesis fails, then $A$ has~$2^{\fc}$ nonisomorphic ultrapowers 
associated with ultrafilters on $\bbN$. 
\end{prop} 

In Proposition~\ref{C.x} it suffices
to assume that the character density of $A$ is $\leq\fc$. 
This does not apply to Proposition~\ref{C.RC} below where the separability assumption
is necessary (cf. the last paragraph of \cite[\S 4]{FaHaSh:Model2} or~\cite{FaPhiSte:Relative}). 

\begin{proof} [Proof of Proposition~\ref{C.x}]Since by the above discussion
both classes are axiomatizable with unstable theories, 
   Theorem~\ref{T1.m} implies
that in all of these cases there are $2^{\fc}$ ultrapowers  
with nonisomorphic unit balls. Therefore the result follows.  
\end{proof} 

In the light of Proposition~\ref{C.x}, 
it is interesting to note that the theory of abelian tracial von Neumann algebras is stable
(\cite[\S 4]{FaHaSh:Model1}). More precisely, a tracial von Neumann algebra $M$
has the property that it has nonisomorphic ultrapowers (and therefore by Theorem~\ref{T1.m}
it has $2^{\fc}$ nonisomorphic ultrapowers)  if and only if it is not of type I. 
This is a consequence of \cite[Theorem~4.7]{FaHaSh:Model1}. 

The following is a quantitative improvement of \cite[Proposition~3.3]{FaHaSh:Model1}, 
confirming a conjecture of 
Sorin Popa in the case when the Continuum Hypothesis fails. 
The intended ultrapower is the tracial ultrapower, and the analogous 
result for norm ultrapower is also true. 

\begin{prop} \label{P.A1} Assume the Continuum Hypothesis fails. Then there
are~$2^{\fc}$ ultrafilters  on $\bbN$ such that the
II$_1$ factors $\prod_{\cU} M_n(\bbC)$ 
are all nonisomorphic. 
\end{prop}

\begin{proof} This is a direct application of Theorem~\ref{T1.m}, 
using $\preceq_\phi$-chains obtained in 
\cite[Lemma~3.2]{FaHaSh:Model1} . 
\end{proof} 

Assume $M$ is  a C*-algebra or a II$_1$ factor and $\cU$ is a nonprincipal
ultrafilter on $\bbN$. Identify $M$ with its diagonal copy inside $\prod_{\cU} M$. 
The \emph{relative commutant} of $M$ inside its ultrapower is defined as 
\[
\textstyle M'\cap \prod_{\cU} M=\{a\in \prod_{\cU} M\colon (\forall a\in M) ab=ba\}. 
\]
Thus the relative commutant is the set of all elements of $\prod_{\cU} M$ 
realizing the  \emph{relative commutant type} of $M$, consisting of
all conditions of the form $\|xb-bx\|=0$, for $b\in M$. 
(Here $\|\cdot\|$ stands for $\|\cdot\|_2$ in case when $M$ is a II$_1$ factor.)
The relative commutant is a C*-algebra (II$_1$ factor, respectively) and 
it is fair to say that most applications of ultrapowers in operator algebras are applications
of relative commutants. A relative commutant is said to be \emph{trivial} if it is equal to 
 the center of $M$. 
From a model-theoretic point of view, a relative commutant
is a submodel consisting of all realizations of a definable type over $M$. 

The original  motivation for the work in   \cite{Fa:Relative}, \cite{FaHaSh:Model1} and 
\cite{FaHaSh:Model2} came from  the question whether all relative commutants of a
given operator algera in its ultrapowers associated with ultrafilters on $\bbN$ 
are isomorphic. 
This was asked by 
Kirchberg in the case of  C*-algebras and McDuff in the case of  II$_1$-factors. 
Here is a quantitative improvement to the answer to these questions given in the 
above references. 

\begin{prop} \label{C.RC} 
Assume $A$ is a separable infinite-dimensional C*-algebra or a separably 
acting II$_1$-factor. 
If the Continuum Hypothesis fails, then~$A$ has 
 $2^{\fc}$ nonisomorphic relative commutants in ultrapowers
associated with ultrafilters on $\bbN$. 
\end{prop} 

\begin{proof} In \cite[Lemma~3.2 (3)]{FaHaSh:Model1} and 
 \cite[Lemma~3.4]{FaHaSh:Model1} it was proved that the  
relative commutant type of a II$_1$ factor has the 
order property (cf. \cite[Example~4.8 (1)]{FaHaSh:Model1}), witnessed by $\psi$ 
given in the introduction to \S\ref{S.appl.OA}.  
In \cite[Lemma~5.2]{FaHaSh:Model1} it was proved that the 
 relative commutant type of any infinite-dimensional C*-algebra has the order
property, witnessed by $\phi$ given in the introduction to \S\ref{S.appl.OA}. 
Hence applying Proposition~\ref{P.local.1.m} concludes the proof. 
\end{proof} 

By $\cB(H)$ we shall denote the C*-algebra of all bounded linear operators 
on an infinite-dimensional, separable, complex Hilbert space $H$. 
In \cite{FaPhiSte:Relative} it was proved that that for certain ultrafilters on $\bbN$ 
the relative commutant of $\cB(H)$ in $\prod_{\cU} \cB(H)$ is nontrivial. 
These ultrafilters exist in ZFC. It was also proved in 
\cite{FaPhiSte:Relative} that the relative commutant of $\cB(H)$ in an ultrapower
associated to a selective ultrafilter is trivial.  
Therefore CH implies that not all relative commutants of 
$\cB(H)$ in its ultrapowers associated with ultrafilters on $\bbN$ are isomorphic. 
This fact motivated Juris Stepr\=ans and the first author to ask whether this statement
can be proved in ZFC. Since $\cB(H)$ is not a separable C*-algebra, the following 
 is not a consequence
of Proposition~\ref{C.RC}. 

\begin{prop} \label{P.B(H)} 
Assume that the Continuum Hypothesis fails. 
Then $\cB(H)$ has $2^{\fc}$ nonisomorphic relative commutants associated with 
its ultrapowers. 
\end{prop} 

\begin{proof} We shall apply Proposition~\ref{P.local.2.m}. 
The following construction borrows some ideas from the proof 
of  \cite[Theorem~3.3 and Theorem~4.1]{FaPhiSte:Relative}. 
Let $\bbF^{<\bbN}$ be the countable set of all finite sequences of 
nonincreasing functions
$h \colon \bbN \to \bbQ \cap [0, 1]$ that are eventually zero
and such that $h (0) = 1$.
We shall construct a filter $D$ on $\bbF^{<\bbN}$.
For $f$ and $g$ in $\bbR^{\bbN}$ write $\|f-g\|_\infty=\sup_i |f(i)-g(i)|$. 
For $f\colon \bbN\nearrow\bbN$ and $m\in \bbN$
let
$X_{f,m}$ be the set of all $k$-tuples  $\langle h_0,h_1,\dots h_{k-1}\rangle$ 
in $ \bbF$ such that 
\begin{enumerate}
\item\label{IC.1}  $k\geq m$, 
\item\label{IC.2}  $ \max_{i<k} \|h_i-h_i\circ f\|_{\infty} \leq 1/m$, 
\item\label{IC.3}  $ h_i(j)\leq h_{i+1}(j)$ for all $i<k-1$ and all $j$, 
\item\label{IC.4}  for all $i<k-2$ there is $j\in \bbN$ such that $h_i(j)=0$ and $h_{i+1}(j)=1$. 
\end{enumerate}
We claim that $X_{f,m}$ is always infinite. This is essentially a consequence of 
the proof of \cite[Lemma~3.4]{FaPhiSte:Relative} but we shall sketch a proof. 
Fix a sequence $n(j)$, for $j\in \bbN$, such that $n(l+1)\geq f(n_l)$ for all $l$. 
For $Z\subseteq \bbN$ by $\chi_Z$ we denote the characteristic 
function of $Z$. For  $i<k$ set 
\[
h_i=\chi_{[0,mi)}+\sum_{l= im} ^{(i+1)m-1}\frac {(i+1)m-l}m \chi_{ [n(l),n(l+1))}.
\]
A straightforward computation shows that  $\langle h_0,h_1,\dots , h_{k-1}\rangle \in X_{f,m}$. 
Since 
$
X_{f, m} \cap X_{g, n}
 \supseteq X_{\max(f, g), \, \max(m,n)}$, 
the collection of all $X_{f, m}$,
for $f \colon \bbN \nearrow\bbN$ and $\e > 0$,
has the  finite intersection property.
Since the filter generated by these sets is analytic, proper, 
 and includes all cofinite sets, it is meager (see e.g., \cite{BarJu:Book}). 
Fix a basis $e_j$, for $j\in \bbN$, of $H$. 
For $h\colon \bbN\to [0,1]$ define a positive operator $a_h$ in $\cB(H)$ via
\[
\textstyle a_h=\sum_{j\in \bbN} h(j) e_j. 
\]
In other words, $a_h$ is the operator with the eigenvalues $h(j)$ corresponding to the
eigenvectors $e_j$. Fix an enumeration $\bbF^{<\bbN}=\{s_i\colon i\in\bbN\}$.  
Let $\bh$ be a function from $\bbN$ into the finite sequences of positive
operators in the unit ball of $\cB(H)$ defined by 
 $
 \bh(i)=\langle a_h\colon h\in s_i\rangle$. 
 With 
\[
\phi(x,y)=\|xy-y\|
\]
conditions \eqref{IC.3} and \eqref{IC.4} above imply that each $\bh(i)$  is a $\phi$-chain. 

Let $\bt$ be the relative commutant type of $\cB(H)$, i.e., the set of all conditions
of the form $\|ax-xa\|<\e$ for $a$ in the unit ball of $\cB(H)$ and $\e>0$. 
Let $\bt_0$ be a finite subset of $\bt$, let $\e>0$,   and let $a_0,\dots, a_{k-1}$ list all 
elements of~$\cB(H)$ occurring in $\bt_0$. Let $\delta=\e/6$. 
  \cite[Lemma~4.6]{FaPhiSte:Relative} implies that there are $g_0$ and $g_1$
  such that for each $i<k$ we can write $a_i=a_i^0+a_i^1+c_i$ so that 
  \begin{enumerate}
  \item $a_i^0$ commutes with $a_h$ for every $h$ that is constant on every interval
  of the form $[g_0(m), g_0(m+1))$, 
  \item $a_i^1$ commutes with $a_h$ for every $h$ that is constant on every interval
  of the form $[g_1(m), g_1(m+1))$, and
  \item $\|c_i\|<\delta$. 
  \end{enumerate}
  Then for $i<k$, $j\in X_{g_0,\delta}\cap X_{g_1,\delta}$, and $h$ an entry of $\bh(j)$ we have
  \[
  [a_i,a_{h}]=
  [a_i^0,a_{h}]+
    [a_i^1,a_{h}]+
      [c_i,a_{h}]
      \]
and since $\|a_i^0\|$, $\|a_i^1\|$ and $\|a_h\|$ are all $\leq 1$
we conclude that $\|[a_i,a_h]\|<6\delta$.   

Therefore  $a_h$  realizes  $\bt_0$, and   
  Proposition~\ref{P.local.2.m} implies that for every linear order $I$ of cardinality $\fc$
  there is an ultrafilter~$\cU$ such that  $\prod_{\cU} \cB(H)$
  contains a $\phi$-chain $\cC$ isomorphic to $I$ which is $(\aleph_1,\phi)$-skeleton 
  like and included in the relative commutant of $\cB(H)$. 
  Since $\phi$ is quantifier-free, $\cC$ remains a $\phi$-chain in 
  the relative commutant 
  $\cB(H)'\cap \prod_{\cU} \cB(H)$. 
  Since $\cC$ is  $(\aleph_1,\phi)$-skeleton 
 like in $\prod_{\cU} \cB(H)$, it is $(\aleph_1,\phi)$-skeleton 
 like in the substructure.  Using 
  Lemma~\ref{L.many-invariants}, 
   Lemma~\ref{L.J}, 
  Lemma~\ref{L.disj.1} 
and a counting  counting argument
as  in the proof of Theorem~\ref{T1+}  we conclude the proof. 
  \end{proof}

\subsection{Concluding remarks}

Before Theorem~\ref{T1} was proved
the following test question was asked in a preliminary 
version of~\cite{FaHaSh:Model2}: Assume $A$ and $B$ are countable models with 
 unstable theories. Also assume $\cU$ and $\cV$ are ultrafilters on $\bbN$ 
 such that $\prod_{\cU} A\not\cong \prod_{\cV} A$. Can we conclude that 
 $\prod_{\cU} B\not\cong \prod_{\cV} B$? 
 A positive answer would, together with \cite[\S3]{KShTS:818},
 imply Theorem~\ref{T1}. 
 However, the answer to this question is consistently negative. 
 Using the method of  \cite{Sh:509}
one can show that in the model obtained there there are   countable graphs
$G$ and $H$ and  ultrafilters  $\cU$ and $\cV$ on $\bbN$   such that 
 $\prod_{\cU} G$, $\prod_{\cV} G$ and $\prod_{\cV} H$ are saturated but $\prod_{\cU} H$ is not. 
 This model has an even more remarkable property:  Every automorphism of
 $\prod_{\cU} H$ lifts to an automorphism of $H^{\bbN}$.
 An interesting and related application of \cite{Sh:509} was recently given in 
 \cite{LuTh:Automorphism}.  

The method of the present paper was adapted to a non-elementary class of all approximately 
matricial
 (shortly AM) C*-algebras in \cite{FaKa:NonseparableII}. 
 A C*-algebra is AM if and only if it is an inductive limit of 
 finite-dimensional matrix algebras.  
 In \cite{FaKa:NonseparableII}  it was proved that in every uncountable
character density~$\lambda$ there are $2^\lambda$ nonisomorphic AM algebras. 


\begin{thebibliography}{10}

\bibitem{BarJu:Book}
T.~Bartoszynski and H.~Judah, \emph{Set theory: on the structure of the real
  line}, A.K. Peters, 1995.

\bibitem{BYBHU}
I.~Ben~Ya'acov, A.~Berenstein, C.W. Henson, and A.~Usvyatsov, \emph{Model
  theory for metric structures}, Model Theory with Applications to Algebra and
  Analysis, Vol. II (Z.~Chatzidakis et~al., eds.), Lecture Notes series of the
  London Math. Society., no. 350, Cambridge University Press, 2008,
  pp.~315--427.

\bibitem{BYHJR}
I.~Ben~Ya'acov, W.~Henson, M.~Junge, and Y.~Raynaud, \emph{Preliminary report -
  {vNA} and {NCP}}, preprint, 2008.

\bibitem{Black:Operator}
B.~Blackadar, \emph{Operator algebras}, Encyclopaedia of Mathematical Sciences,
  vol. 122, Springer-Verlag, Berlin, 2006, Theory of $C\sp *$-algebras and von
  Neumann algebras, Operator Algebras and Non-commutative Geometry, III.

\bibitem{ChaKe}
C.~C. Chang and H.~J. Keisler, \emph{Model theory}, third ed., Studies in Logic
  and the Foundations of Mathematics, vol.~73, North-Holland Publishing Co.,
  Amsterdam, 1990.

\bibitem{Do:Ultrapowers}
A.~Dow, \emph{On ultrapowers of {B}oolean algebras}, Topology Proc. \textbf{9}
  (1984), no.~2, 269--291.

\bibitem{ElHaScTh}
P.~Ellis, S.~Hachtman, S.~Schneider, and S.~Thomas, \emph{Ultraproducts of
  finite alternating groups}, RIMS Kokyuroku \textbf{No. 1819} (2008), 1--7.

\bibitem{Fa:Relative}
I.~Farah, \emph{The relative commutant of separable {C*}-algebras of real rank
  zero}, Jour. Funct. Analysis \textbf{256} (2009), 3841--3846.

\bibitem{FaHaSh:Model1}
I.~Farah, B.~Hart, and D.~Sherman, \emph{Model theory of operator algebras {I}:
  {S}tability}, preprint, arXiv:0908.2790, 2009.

\bibitem{FaHaSh:Model2}
\bysame, \emph{Model theory of operator algebras {II}: {M}odel theory},
  preprint, 2009.

\bibitem{FaKa:NonseparableII}
I.~Farah and T.~Katsura, \emph{Nonseparable {UHF} algebras {II}:
  Classification}, in preparation, 2009.

\bibitem{FaPhiSte:Relative}
I.~Farah, N.C. Phillips, and J.~Stepr\=ans, \emph{The commutant of {$L(H)$} in
  its ultrapower may or may not be trivial}, Math. Annalen (to appear),
  http://arxiv.org/abs/0808.3763v2.

\bibitem{GeHa}
L.~Ge and D.~Hadwin, \emph{Ultraproducts of {$C\sp *$}-algebras}, Recent
  advances in operator theory and related topics (Szeged, 1999), Oper. Theory
  Adv. Appl., vol. 127, Birkh\"auser, Basel, 2001, pp.~305--326.

\bibitem{KShTS:818}
Linus Kramer, Saharon Shelah, Katrin Tent, and Simon Thomas, \emph{Asymptotic
  cones of finitely presented groups}, Advances in Mathematics \textbf{193}
  (2005), 142--173, math.GT/0306420.

\bibitem{Ku:Book}
K.~Kunen, \emph{Set theory: An introduction to independence proofs},
  North--Holland, 1980.

\bibitem{LuTh:Automorphism}
P.~L\"ucke and S.~Thomas, \emph{Automorphism groups of ultraproducts of finite
  symmetric groups}, preprint, 2009.

\bibitem{Pe:Hyperlinear}
V.~Pestov, \emph{Hyperlinear and sofic groups: a brief guide}, Bulletin of
  Symbolic Logic \textbf{14} (2008), 449--480.

\bibitem{Sh:e}
Saharon Shelah, \emph{Non--structure theory}, vol. accepted, Oxford University
  Press.

\bibitem{Sh:c}
\bysame, \emph{Classification theory and the number of nonisomorphic models},
  Studies in Logic and the Foundations of Mathematics, vol.~92, North-Holland
  Publishing Co., Amsterdam, xxxiv+705 pp, 1990.

\bibitem{Sh:509}
\bysame, \emph{Vive la diff\'erence iii}, Israel Journal of Mathematics
  \textbf{166} (2008), 61--96, math.LO/0112237.

\bibitem{Th:On}
S.~Thomas, \emph{On the number of universal sofic groups}, preprint, 2009.

\end{thebibliography}

\def\germ{\frak} \def\scr{\cal} \ifx\documentclass\undefinedcs
  \def\bf{\fam\bffam\tenbf}\def\rm{\fam0\tenrm}\fi 
  \def\defaultdefine#1#2{\expandafter\ifx\csname#1\endcsname\relax
  \expandafter\def\csname#1\endcsname{#2}\fi} \defaultdefine{Bbb}{\bf}
  \defaultdefine{frak}{\bf} \defaultdefine{=}{\B} 
  \defaultdefine{mathfrak}{\frak} \defaultdefine{mathbb}{\bf}
  \defaultdefine{mathcal}{\cal}
  \defaultdefine{beth}{BETH}\defaultdefine{cal}{\bf} \def\bbfI{{\Bbb I}}
  \def\mbox{\hbox} \def\text{\hbox} \def\om{\omega} \def\Cal#1{{\bf #1}}
  \def\pcf{pcf} \defaultdefine{cf}{cf} \defaultdefine{reals}{{\Bbb R}}
  \defaultdefine{real}{{\Bbb R}} \def\restriction{{|}} \def\club{CLUB}
  \def\w{\omega} \def\exist{\exists} \def\se{{\germ se}} \def\bb{{\bf b}}
  \def\equivalence{\equiv} \let\lt< \let\gt>
\providecommand{\bysame}{\leavevmode\hbox to3em{\hrulefill}\thinspace}
\providecommand{\MR}{\relax\ifhmode\unskip\space\fi MR }
\providecommand{\MRhref}[2]{%
  \href{http://www.ams.org/mathscinet-getitem?mr=#1}{#2}
}
\providecommand{\href}[2]{#2}

\end{document}